\documentclass[12pt, reqno]{amsart}
\usepackage{amsmath, amsthm, amscd, amsfonts, amssymb, graphicx, color}
\usepackage[bookmarksnumbered, plainpages]{hyperref}

\usepackage{tikz}
\usetikzlibrary{arrows, shapes,positioning}
\usepackage[utf8]{inputenc}
\newtheorem{theorem}{Theorem}[section]
\newtheorem{lemma}[theorem]{Lemma}

\newtheorem{corollary}[theorem]{Corollary}
\theoremstyle{definition}
\newtheorem{definition}[theorem]{Definition}
\newtheorem{example}[theorem]{Example}

\theoremstyle{remark}
\newtheorem{remark}[theorem]{Remark}
\numberwithin{equation}{section}

\begin{document}
	\setcounter{page}{1}
	
	\title[Zip shift space]{Zip shift space}
	
	\author[S. Lamei, P. Mehdipour]{Sanaz Lamei$^{*1}$, Pouya Mehdipour$^2$}
	
	\address{$^{1}$ Faculty of Mathematical sciences , University of Guilan,  P. O. Box 1914, Rasht, Iran.}
	\email{lamei@guilan.ac.ir}
	\address{$^{2}$ Departamento de MatemÃ¡tica,Universidade Federal de ViÃ§osa, DMA- Brazil.}
	\email{Pouya@ufv.br}
	
	%\dedicatory{This paper is dedicated to Professor ABCD}
	
	\subjclass[2010]{Primary 37B10; Secondary 37D05, 37C29.}
	
	\keywords{ zip shift, endomorphim horseshoe, endomorphim homoclinic orbit.}
	
	\date{Received: xxxxxx; Revised: yyyyyy; Accepted: zzzzzz.
		\newline \indent $^{*}$ Corresponding author}
%	\maketitle
	\tableofcontents
	\begin{abstract}
		We introduce a new extension in symbolic dynamics on two sets of alphabets, called the zip shift space. In finite case, 
		it represents a finite-to-1 local homeomorphism called zip shift map. Such extension, offers a conjugacy between some
		endomorphisms and some zip shift map over two-sided space with finite sets of alphabets. As an application, the topological conjugacy
		of an N-to-1 uniformly hyperbolic horseshoe map with a zip shift map and its orbit structure is investigated. Moreover, the pre-image studies
		over zip shift space and the concepts of stable and unstable sets and homoclinic orbits, with a precise description for N-to-1 horseshoe are illustrated.
	\end{abstract} 
	\maketitle
	
	\section{Introduction}
	Symbolic dynamics is a powerful implement to verify a worth amount of dynamical properties in smooth and topological dynamics. It provides a discrete modeling of the system.
	Any state of such modeling is represented by finite (or infinite) set of alphabets. 
	The evolution of dynamics is displayed by the shift transformation over the space of one-sided  or two-sided sequences of configurations.
	
	The idea of symbolic dynamics goes back to the studies of Hadamard (1989) on geodesic flows of surfaces with negative curvature \cite{H2}. 
	After that the theory was further developed by the attempts of  mathematicians such as Morse, 
	Artin, Hedlund, Birkhoff, Levinson, Cartwright, Littlewood etc. \cite{A}, \cite{AF}, \cite{LM}.
	Around 1970, the two-sided symbolic dynamics and the Markov partition was developed for hyperbolic toral automorphisms and Axiom A
	diffeomorphisms (including Anosov diffeomorphisms), by  
	Adler, Weiss, Anosov, Sinai, and Bowen \cite{AF}, \cite{AW},  \cite{B1}, \cite{S2}. 
	In particular, R. Bowen used  symbolic dynamics and Markov partition   to study the Axiom A diffeomorphisms and it
	allowed the techniques of statistical mechanics to be applied to dynamical systems. Also S. Smale  
	developed a method to determine the chaotic behavior of dynamical systems by giving a symbolic representation 
	to the horseshoe map. Currently, this powerful method is improved in many areas of mathematics, such as ergodic theory,
	complex dynamics, partially hyperbolic diffeomorphisms, billiards, etc., and has significant applications in
	other research areas such as information theory, cellular automaton, statistical mechanics, quantum chaos,  thermodynamics, etc \cite{AA}, \cite{S1}, \cite{R}. 
	
	In the context of invertible dynamics, the existence of a Markov partition \cite{B1}, \cite{LM}
	provides a topological (or measure theoretical) conjugacy with some two-sided shift homeomorphism and offers convenient dynamical information. 
	In non-invertible case, 
	when the system admits a Markov partition, one 
	can associate a semi-conjugacy to a shift map over one-sided shift space. In addition, another semi-conjugacy can be obtained via the natural extension  (inverse limit) of the map
	over the (two-sided) orbit space. 
	It is noteworthy that recently,  E. Ermerson, Y. Lima and P. Mauricio 
	\cite{ELM}, used  the methods introduced in \cite{O1}, \cite{S}, to construct countable Markov partitions for the natural extension of some non-invertible 
	non-uniformly hyperbolic systems.
	This construction allowed to verify some ergodic and topological properties compatible with the inverse limit of the original map.
	In fact, there is lots of important data, which is accessible only for the natural extension of the map,
	and not necessarily for the original dynamics. 
	To the best of our knowledge, 
	outside the natural extension 
	of the map  which is a symbolic space over uncountably many symbols \cite{BJ}, \cite{P}, \cite{QXZ},
	no attempt has been done to develop a two-sided symbolic dynamics on finite alphabets for non-invertible dynamics. 
	
	In this work, we aim to introduce the zip shift maps over  zip shift spaces with  finite alphabets, which are extensions of two-sided shift maps over shift spaces. 
	The zip shift map is defined as a local homeomorphism on the zip shift space which is the set of bi-infinite sequences 
	over two sets of alphabets. 
	This is because for non-invertible dynamics, although the forward dynamics is induced by the backward one, yet it can be a completely different dynamics, including important data. 
	%*
	For instance, in Section \ref{ho}
	we compare the 2-to-1 horseshoe map
	and 
	the map $g(x)=4x$ modulo 1.  As it is known, both of  the maps  are semi-conjugate to the full one-sided shift over the full-shift in 4 symbols and the shift 
	map can not distinguish these two dynamics. 
	We show that these systems can be identified, by conjugation with different zip shift maps.
	From a measure theoretical point of view, in \cite{LM1},\cite{MM}, we introduce the notion of $(m,l)$-Bernoulli property as 
	an extension to the previously known two-sided Bernoulli property.  It illustrates the fundamental role of investigating two-sided ergodic properties in 
	identifying one-sided Bernoulli maps. 
	From a dynamical system point of view we wish that, the zip shift map provides a new approach to
	the classification of local homeomorphisms and to a better comprehension of the statistical properties of local diffeomorphisms.
	
	In addition, in this work, we present a concise construction of an $N$-to-1 horseshoe set, which  is an example of admitting a good Markov partition for endomorphisms. 
	In such cases 
	the zip shift map acts like a zipper by taking special bi-infinite sequences and merge or 
	collapse the entries of negative indices and create entries with new alphabet  to describe the forward dynamics (Definition \ref{ZS}).
	When a system admits a `good' Markov partition or even a `good' labeling for its different states, 
	(the labeled states in forward and backward should map together by dynamics) it is enough to represent the map by a conjugate zip shift map.
	
	In what follows,  Section \ref{sec:2} of this paper contains the  introduction to the zip shift maps  over zip shift spaces and also some basic definitions, properties and
	generalizations.

	Section \ref{sec3pre} proceeds with the study of pre-images of points under zip shift maps. In Section \ref{sec:3.2}, 
	for the case when the zip shift space is sofic, we defined a labeled graph where the bi-infinite walks on the graph represent the elements of the zip shift space. 
	Example \ref{eg p } shows that in general case,  even for a simple zip shift space, it is not easy to find the exact pre-images of the points. 
	In Theorem \ref{main2},  we provide a way 
	to find all the pre-images of a given point, up to sofic zip shift spaces.

	A generalization of the Smale horseshoe to an $N$-to-1 case, is given in Section \ref{ho}. Two natural Markov partitions, 
	one for forward iterates and another for the backward iterates are stated. These partitions are mapped together by horseshoe 
	map and give rise to a zip shift space. In Theorem \ref{4.2} the corresponding conjugacy between a zip shift map and the horseshoe map is stated. 
	In Example  \ref{ex:4.3} the 2-to-1 horseshoe map and the map $g(x)=4x$ modulo 1 are compared.
	
	In Section \ref{long},
	the stable-unstable sets; periodic, pre-periodic and  homoclinic  points and their orbits for the zip shift map, are defined.
	Furthermore, as an application in Subsection \ref{PPH-ho}
	the stable-unstable sets for the  $N$-to-1 horseshoe, also, their periodic and  homoclinic points as well as their orbits, are investigated. These studies asserts many new data on the 
	behavior and the importance of the role of the pre-periodic points (see Definition \ref{pre-p}) homoclinic points and their orbits. 
	
	\section{Zip shift spaces}\label{sec:2}
	%\large
	Consider two finite sets of alphabets $\mathcal{A}=\{a,\ldots,\,b\}$ and $\mathcal{A}'=\{a',\ldots,\,b'\}$. 
	Let  $\Sigma_{\mathcal{A}'}\subseteq \{(x_i)_{i\in \mathbb{N}\cup\{0\}},\, x_i\in \mathcal{A}'\}$ be a one-sided  shift space.
	A finite sequence $w=w_{[k,\ell]}=x_k\, x_{k+1}\, \ldots\, x_{\ell}$  which occurs in some point $(x_i)_{i\in \mathbb{N}\cup\{0\}}\in \Sigma_{\mathcal{A}'}$
	is called a \textit{word} or \textit{block} of length $|w|=\ell-k+1$. 
	Let $\beta_n^{\mathcal{A}'}=\beta_n^{\mathcal{A}'}(\Sigma_{\mathcal{A}'})$ be 
	the set of all admissible words of length $n$ and $\beta^{\mathcal{A}'}=\beta^{\mathcal{A}'}(\Sigma_{\mathcal{A}'})=\bigcup_{n\geq 1}\beta_n^{\mathcal{A}'}$. Obviously, $\beta_1^{\mathcal{A}'}=\mathcal{A}'$.
		\subsection{Definition of zip shift space and zip shift map}
	Let  $\varphi_n: \beta^{\mathcal{A}'}_n\longrightarrow \mathcal{A}$ be a transition factor map which is onto
	(not necessarily invertible)  with domain $\beta^{\mathcal{A}'}_n$. 
	Associate to $\varphi_n$ a 0-1 transition matrix $T$ such that for the word $a_1'\ldots a_n'\in \beta^{\mathcal{A}'}_n$ and $a\in \mathcal{A}$, $T(a_1'\ldots a_n',\,a)=1$ if and only if $\varphi_n(a_1'\ldots a_n')=a$.

	Let $Y=\{y=(y_i)_{i\in \mathbb{Z}},\, y_i\in \mathcal{A}'\}$ be a two-sided shift space with its shift map $\sigma_Y$. To any point $y\in Y$,  correspond a 
	point $x=(x_i)_{i\in \mathbb{Z}}=(\ldots,x_{-2},\,x_{-1};\,x_0,\,x_1,\,x_2,\ldots)$ such that
	\begin{equation}\label{ZS}
		x_i=\left\{\begin{tabular}{ll}
			$y_i\in \mathcal{A}' \hspace{7mm} $& $\forall i\geq 0$\\
			$\varphi_n (y_{i}\ldots y_{i+n-1})\in \mathcal{A} \hspace{7mm} $& $\forall i<0$. 
		\end{tabular}\right.
	\end{equation}
	According to set $Y$, define the \textit{zip shift space} as the set
	$$\Sigma_{\mathcal{A},\mathcal{A}'}:= \{x=(x_i)_{i\in \mathbb{Z}}\ : x_i \textrm{\ satisfies} \,(\ref{ZS})\}.$$
	If $Y$ is a full shift space, then $\Sigma_{\mathcal{A},\mathcal{A}'}$  will be called a \textit{full zip shift space}.
	Let $\Sigma_{\mathcal{A}'}=\{(x_0,\, x_1,\ldots)\ :  (x_i)_{i\in \mathbb{Z}}\in \Sigma_{\mathcal{A},\mathcal{A}'}\}$ and  
	$\Sigma_{\mathcal{A}}^-=\{( \ldots, x_{-2},\, x_{-1})\ :  (x_i)_{i\in \mathbb{Z}}\in \Sigma_{\mathcal{A},\mathcal{A}'}\}$. 
	From here on, we denote $\Sigma_{\mathcal{A}}^-$ by $\Sigma_{\mathcal{A}}$.
	The set $\Sigma_{\mathcal{A}}$ is a factor  of $\Sigma_{\mathcal{A}'}$ due to the map $\varphi_n$.
	Define the \textit{zip shift map} on $\Sigma_{\mathcal{A},\mathcal{A}'}$ with same notation for shift maps as follows
	$$(\sigma(x))_i=\left\{\begin{tabular}{ll}
		$\varphi_n(x_{0}\ldots x_{n-1})\hspace{7mm} $& if $i= -1$\\
		$x_{i+1}\hspace{7mm} $& otherwise. 
	\end{tabular}\right.$$ 
	The zip shift map is a local homeomorphism  generalization of a two-sided shift homeomorphism.  
	If one let $\mathcal{A}=\mathcal{A}'$ and $\varphi_1=Id$ then the zip shift map is  a  two-sided shift map.
	
	Any subspace $\Sigma$ of $\Sigma_{\mathcal{A},\mathcal{A}'}$ which is closed under $\sigma$ is
	called a \textit{sub zip shift space} of $\Sigma_{\mathcal{A},\mathcal{A}'}$.
Now $\beta_n^{\mathcal{A},\mathcal{A}'}=\beta^{\mathcal{A},\mathcal{A}'}_n(\Sigma)$ is the set of all words of length $n$ in $\Sigma$. 
	The set $\beta^{\mathcal{A},\mathcal{A}'}=\beta^{\mathcal{A},\mathcal{A}'}(\Sigma)=\bigcup_{n\geq 1}\beta^{\mathcal{A},\mathcal{A}'}_n(\Sigma)$
	is  the set of all \textit{admissible} words or the \textit{language} of $\Sigma$.
	Let $\beta_n^{\mathcal{A}}$ be the subset of $\beta^{\mathcal{A},\mathcal{A}'}(\Sigma)$ with alphabets in  $\mathcal{A}$  and 
	$\beta^{\mathcal{A}}=\bigcup_{n\geq 1}\beta_n^{\mathcal{A}}$. For $w'=a_i' a_{i+1}'\ldots a_{i+m}'\in \beta^{\mathcal{A}'}$ and  $m>n$, 
	we simply extend the definition of $\varphi_n$ as
	\begin{equation}\label{extend}
		\varphi_n(w')=\varphi_n(a_i'\ldots a_{i+n-1}')\varphi_n(a_{i+1}'\ldots a_{i+n}')\ldots \varphi_n(a_{i+m-n+1}'\ldots a_{i+m}')\in \beta^{\mathcal{A}}.
	\end{equation}
	
	Let  $\Sigma_{\mathcal{A},\mathcal{A}'}$ be a full zip shift space. For $\Sigma\subset \Sigma_{\mathcal{A},\mathcal{A}'}$  with admissible words $\beta^{\mathcal{A},\mathcal{A}'}$, 
	the words in $\beta^{\mathcal{A},\mathcal{A}'}(\Sigma_{\mathcal{A},\mathcal{A}'})\setminus \beta^{\mathcal{A},\mathcal{A}'}(\Sigma)$ will not appear in any sequences in $\Sigma$.
	Let $\mathcal{F}$ be a minimal set of words (a set with fewest and shortest words)  such that no word in $\beta^{\mathcal{A},\mathcal{A}'}(\Sigma)$ contains a word of $\mathcal{F}$. This set is called  
	the set of \textit{forbidden blocks} of  $\Sigma$. Let $\mathcal{F}^{\mathcal{A}}$ and $\mathcal{F}^{\mathcal{A}'}$ 
	be the subsets of $\mathcal{F}$ 
	%whose words  are consisted of the alphabets in $\mathcal{A}$ and $\mathcal{A}'$ which are in fact 
	determining the forbidden blocks of the
	sets $\Sigma_{\mathcal{A}}$ and $\Sigma_{\mathcal{A}'}$ respectively.  
	Note that if  $w'=a_1'a_2'\ldots a_m'\in \mathcal{F}$, 
	then $\varphi_n$ will not be defined for any word containing $w'$.
	
	Recall that a shift space is a subshift of finite type (SFT), if   there are only finite forbidden blocks with finite length.  For any SFT, there exists a $K$ such that the SFT is
	a $K$-step Markov space and the length of forbidden blocks is less than $K+1$. The image of an SFT under a factor map is called a sofic space \cite{LM}. 
	For a  zip shift space $\Sigma_{\mathcal{A},\mathcal{A}'}$, if $|\mathcal{F}|<\infty$ and the length of longest word in $\mathcal{F}$ is $M+1$, 
	then call $\Sigma_{\mathcal{A},\mathcal{A}'}$  an \textit{$M$-step Markov space}. 
	More precisely,  suppose that $\Sigma_{\mathcal{A}'}$ is  $K$-step  and $\Sigma_{\mathcal{A}}$ is   $N$-step.
	Then for $M=\max \{N,\,K\}$, $\Sigma_{\mathcal{A},\mathcal{A}'}$ is an  $M$-step Markov space. 
	In this case one can associate 3 matrices such that they all together represent the $M$-step space $\Sigma_{\mathcal{A},\mathcal{A}'}$.
	As for shift  spaces, $\Sigma_{\mathcal{A}'}$ can be represented by an adjacency matrix $A'$ such that for 
	words $v'=a_1'\ldots a_{K}',\,w'=b_1'\ldots b_{K}'\in  \beta^{\mathcal{A}'}_{K}$,  $A'(v',\ w')=1$  if and only if
	$a_1'\ldots a_{K-1}'=b_2'\ldots b_{K}'$ and $a_1'\ldots a_{K}'b_{K}'\in \beta^{\mathcal{A}'}_{K+1}$. 
	The matrices $A$ and $T$ are defined as follows. For words $v=a_1\ldots a_{N},\,w=b_1\ldots b_{N}\in  \beta^{\mathcal{A}}_{N}$,  $A(v,\ w)=1$  if and only if
	$a_1\ldots a_{N-1}=b_2\ldots b_{N}$ and $a_1\ldots a_{N}b_{N}\in \beta^{\mathcal{A}}_{N+1}$. 
	Also, $T(v,\, v')=1$ if and only if $vv'\in\beta^{{\mathcal{A},\mathcal{A}'}}_{N+K} $. We call $A$ and $A'$ the {\it adjacency} matrices due to alphabets $\mathcal{A}$ and $\mathcal{A}'$. The matrix $T$ can be defined by using the words of maximum length $M+1$ but it would be a larger matrix.

	The space  $\Sigma_{\mathcal{A},\mathcal{A}'}$ is an SFT only if both $\Sigma_{\mathcal{A}}$ and $\Sigma_{\mathcal{A}'}$ are SFTs.
	For an SFT or sofic space $\Sigma_{\mathcal{A}'}$, the factor space $\Sigma_{\mathcal{A}}$ can be either SFT  or sofic.
	If both $\Sigma_{\mathcal{A}}$ and $\Sigma_{\mathcal{A}'}$ are sofic spaces (any SFT is sofic), then call $\Sigma_{\mathcal{A},\mathcal{A}'}$ a \textit{sofic space}. Such   space 
	is called \textit{strictly sofic} if it is not an SFT.

	If $\Sigma_{\mathcal{A},\mathcal{A}'}$ does not depend on time, then in some cases it can be represented by graphs. 
	Let $\Sigma_{\mathcal{A}'}$ be a sofic (not necessarily strictly sofic) space  with labeled graph  $\mathcal{G}=(G,\, \mathcal{L})$
	(see \cite{LM} for definitions). Let $y=(y_i)_{i\in \mathbb{Z}}$ be a bi-infinite walk on $\mathcal{G}$.
	Then $x=(x_i)_{i\in \mathbb{Z}}$ is a point in $\Sigma_{\mathcal{A},\mathcal{A}'}$ if $x_i=y_i$ for $i\geq 0$ and $x_i=\varphi_n(y_i,\ldots y_{i+n})$ for $i\leq -1$.
	\begin{example}\label{eg 1}
		Let $G$ be the left graph in Figure \ref{F1}.
		Let    $\mathcal{A}'=\{1,2,3\}$ be the alphabet set   and  $\mathcal{F}=\{11,\, 13,\, 21,\, 33\}$ be   the forbidden set. 
		The corresponding labeled graph is the right graph in Figure \ref{F1}. Suppose $\mathcal{A}=\{a,b\}$ with $\varphi_1(1)=a$ and $\varphi_1(2)=\varphi_1(3)=b$.
			A typical sequence in $\Sigma_{\mathcal{A},\mathcal{A}'}$ will be 
		$$
		\begin{tabular}{rcl}
			$x$&=&$(\ldots,x_{-3},\, x_{-2},\, x_{-1};\, x_0,\,x_1,\,x_2,\, x_3,\,x_4,\ldots )$\\
			&=& $(\ldots,\ \ b,\ \ \ a,\ \ \ b;\ \ \ 3,\ \ 1,\ \ 2,\ \ 2,\ \ 3,\ \ \ldots ),$
		\end{tabular} 
		$$
		which is obtained by a walk on labeled graph $\mathcal{G}$.
	\end{example} 
	
%\begin{figure}
%	\centering
%	\begin{tikzpicture}[->, >=stealth', shorten >=1pt, auto, node distance=2cm, thin, main node/.style={circle, draw}]
%		
%		% Left Graph
%		\begin{scope}[local bounding box=left]
%			\node[main node] (1) {$1$};
%			\node[main node] (2) [right of=1] {$2$};
%			\node[main node] (3) [below of=1] {$3$};
%			
%			\path[every node/.style={font=\sffamily\tiny}]
%			(1) edge [right] node [above] {} (2)
%			(2) edge [loop right] node [right] {} (2)
%			(2) edge [bend left] node [below] {} (3)
%			(3) edge [left] node [above] {} (1)
%			(3) edge [right] node [above] {} (2);
%		\end{scope}
%		
%		% Right Graph
%		\begin{scope}[local bounding box=right, xshift=6cm]
%			\node[main node] (4) {};
%			\node[main node] (5) [right of=4] {};
%			\node[main node] (6) [below of=4] {};
%			
%			\path[every node/.style={font=\sffamily\tiny}]
%			(4) edge [right] node [above] {$1$} (5)
%			(5) edge [loop right] node [right] {$2$} (5)
%			(5) edge [bend left] node [below] {$2$} (6)
%			(6) edge [left] node [left] {$3$} (4)
%			(6) edge [right] node [above] {$3$} (5);
%		\end{scope}
%	\end{tikzpicture}
%	\caption{The right graph is the labeled graph $\mathcal{G}$ of the left graph.}
%	\label{F1}
%\end{figure}

\begin{figure}[t]
\includegraphics[width=0.75\textwidth]{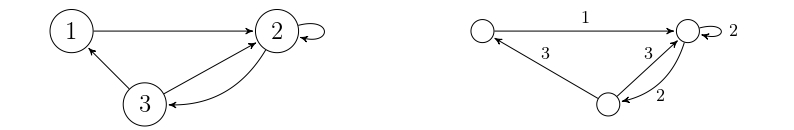}
\caption{The right graph is the labeled graph $\mathcal{G}$ of the left graph.}
\label{F1}
\end{figure}
	\begin{definition}\label{defs}  Let $\Sigma$ be a sub zip shift space  with alphabets $\mathcal{A}$ and $\mathcal{A}'$.
		\begin{itemize}
			\item The  space $\Sigma$  is called \textit{irreducible} if for any $a', b'\in 
			\mathcal{A}'$, there exists a word $w'\in \beta^{\mathcal{A}'}$ such that $a'w'b'\in \beta^{\mathcal{A}'}$. 
			\item A point $(x_i)_{i\in \mathbb{Z}}\in \Sigma$ is  \textit{periodic}  with period $m$ if there exist $p_{1}',\cdots,p_{m-1}'\in \mathcal{A}'$ 
			such that for $i\geq 0$, $x_{mi}=p_0'$, $x_{mi+1}=p_1'$, 
			$\cdots$, $x_{mi+m-1}=p_{m-1}'$  and also, 
			$x_{mi}=\varphi_n(x_0,\cdots, x_{n-1})$, $x_{mi+1}=\varphi_n(x_1,\cdots, x_{n})$, $\cdots$ and 
			$x_{mi+m-1}=\varphi_n(x_{m-1},\cdots, x_{m+n-2})$   for $i\leq -1$. 
			Denote such a periodic point by $p=(\overline{p_0',\,p_1',\,\cdots, p_{m-1}'})$. Let $\mathrm{Per}_m(\Sigma)$ be the
			set of all periodic points with period $m$ and $\mathrm{Per}(\Sigma)=\cup_{m\geq 0}\mathrm{Per}_m(\Sigma)$.
		\end{itemize}
	\end{definition}
	\begin{remark}
		\begin{itemize}
			\item Since $\Sigma$ is closed under $\sigma$, there is no need to verify the irreducibility for the words in $\beta^{\mathcal{A}}$. 
			\item If $\Sigma$ is representable by  graph $G$ or $\mathcal{G}=(G,\mathcal{L})$, then $\Sigma$ is irreducible if and only if $G$ is a strongly connected graph. 
			\item If  matrices $A$, $A'$ and $T$ associate to an SFT zip shift space $\Sigma$, then $\Sigma$ is irreducible if and only if  for any $a',\,b'\in \mathcal{A}'$, there exists $n\geq 0$ with ${A'}^n(a',\,b')\geq 1$.
			\item Suppose the factor map $\varphi_n$ is not one-to-one. Then for some distinct words $w'=a_1'\ldots\,a_{n}'$ and $w''=a_1''\ldots\,a_{n}''$ both in 
			$\beta_{\mathcal{A}'}(\Sigma)$, we may have $\varphi_n(w')=\varphi_n(w'')$.
			For this reason, the repeating part in negative indices of the periodic point may have different length  from the repeating part in non-negative indices.
			For example the point $y=(\ldots, b,b,b;2,3,2,3,2,3,\ldots)=(\overline{2,3})$  is a periodic point in Example \ref{eg 1}. 
			The repeating part in negative and non-negative indices are different. Also, the distinct points $y$ and $z=(\ldots, b,b,b;3,2,2,3,2,2, \ldots)=(\overline{3,2,2})$ have the same negative part.
			\item If $A'$ is the adjacency matrix for an SFT space $\Sigma$, then $\mathrm{trace} (A'^n)$ determines the number of points of period $n$.
		\end{itemize}
	\end{remark}
	
	\subsection{Higher block zip shifts, higher power zip shifts and zip shift sliding block codes}
	The definitions and theorems in this subsection are generalizations of the similar ones for shift spaces defined in \cite{LM}.
	\begin{definition} 
		Let $ \Sigma_{\mathcal{A},\mathcal{A}'}$ be a zip shift space. For some natural number $N$,
		\begin{itemize}
			\item[i)] define the \textit{$N$th higher block code} $h_{N}:\Sigma_{\mathcal{A},\mathcal{A}'}\rightarrow (\beta^{\mathcal{A},\mathcal{A}'}_{N}(\Sigma_{\mathcal{A},\mathcal{A}'}))^{\mathbb{Z}}$ as 
			$$(h_N(x))_{[i]}:=x_{[i,\, i+N-1]}$$
			and the \textit{$N$th higher block zip shift} $\Sigma_{\mathcal{A},\mathcal{A}'}^{[N]}$ as $\Sigma_{\mathcal{A},\mathcal{A}'}^{[N]}=h_{N}(\Sigma_{\mathcal{A},\mathcal{A}'})$.
			\item[ii)] let $\gamma_{N}:\Sigma_{\mathcal{A},\mathcal{A}'}\rightarrow (\beta^{\mathcal{A},\mathcal{A}'}_{N}(\Sigma_{\mathcal{A},\mathcal{A}'}))^{\mathbb{Z}}$ be the \textit{$N$th higher power code}  
			$$(\gamma_N(x))_{[i]}:=x_{[iN,\, iN+N-1]}$$
			and define  the \textit{$N$th higher power zip shift} $\Sigma_{\mathcal{A},\mathcal{A}'}^{N}$  as $\Sigma_{\mathcal{A},\mathcal{A}'}^{N}=\gamma_{N}(\Sigma_{\mathcal{A},\mathcal{A}'})$.
		\end{itemize} 
	\end{definition}
	Note that $h_N(x)$ and $\gamma_{N}(x)$ are defined with respect to the order of successive blocks of length $N$ occurring in $x$.
	The next theorem follows from the definition of zip shift space and Theorems 1.4.3 and 1.4.6 in \cite{LM}. 
	\begin{theorem}\label{h}
		The higher block zip shift code and higher power zip shift code of a zip shift space are  zip shift spaces.
	\end{theorem}
	To define the sliding block code, suppose $X\subset \mathcal{C}^{\mathbb{Z} }$ and $Y\subset \mathcal{C}'^{\mathbb{Z} }$ are two shift spaces with alphabets in
	$\mathcal{C}$ and $\mathcal{C}'$. For these spaces, the block map $\Phi: \beta^{C}_{M+N+1}(X)\rightarrow \mathcal{C}'$ 
	maps the admissible blocks of $X$ with length $M+N+1$ to letters in $\mathcal{C}'$ \cite{LM}. This block map factors the points 
	$x=(\ldots x_{-1},\, x_0,\, x_1,\, \ldots)\in X$ to a point $y=(\ldots y_{-1},\, y_0,\, y_1,\, \ldots)\in Y$ and naturally implies a map 
	$\phi: X\rightarrow Y$ such that 
	$\phi(x)=y$.
		This definition extends to sliding block codes for zip shift spaces. In new definition, we just concider the anticipation $N$.
				\begin{definition}\label{def slide}
			Let $\Sigma_{\mathcal{A},\mathcal{A}'}$ 
			be a 
			zip shift space with transition maps $\varphi_n^{\mathcal{A},\mathcal{A}'}:\beta^{\mathcal{A}'}_n\longrightarrow \mathcal{A}$.
			For two alphabet sets $\mathcal{C}$ and $\mathcal{C}'$, 
			define a  map $$\Psi: \beta^{\mathcal{A},\mathcal{A}'}_{n}\cup \beta^{\mathcal{A},\mathcal{A}'}_{n+1}\rightarrow
			\mathcal{C}\cup \mathcal{C}'$$
			such that 
			\begin{equation}\label{star}
				\Psi(x_{[-1,\, n-1]})=y_{-1}\in \mathcal{C},\ \textrm{and}\  
				\Psi(x_{[0,\, n-1]})=y_{0}\in \mathcal{C}'.
			\end{equation}
			For $m>0$, define $\beta^{\mathcal{C}'}_{m}:= [y_0,\cdots, y_{m-1}]$ induced by iterations of $\Psi$ over subblocks of $x_{[0,m+n-2]}$.
			If there exists a transition map $\varphi_m^{\mathcal{C},\mathcal{C}'}:\beta^{\mathcal{C}'}_m\longrightarrow \mathcal{C}$
			such that 
			\begin{equation}\label{starstar}
				y'_{-1}:=\varphi_m^{\mathcal{C},\mathcal{C}'}([y_0, \cdots , y_{m-1}])=\Psi ([\varphi_n^{\mathcal{A},\mathcal{A}'}[x_0,\cdots , x_{n-1}], x_1, \cdots , x_{n-1}]),
			\end{equation}
			then call $\Psi$ a \textit{block map}. 
			If there exists a map  $\psi:\Sigma_{\mathcal{A},\mathcal{A}'}\rightarrow \Sigma_{\mathcal{C},\mathcal{C}'}$ such that for any $x\in \Sigma_{\mathcal{A},\mathcal{A}'}$, there exists $y\in \Sigma_{\mathcal{C},\mathcal{C}'}$ such that $\psi(x)=y$  and $y_i$ is given by $\Psi$ satisfing in
			\eqref{star} and \eqref{starstar}, then call $\psi$  the \textit{zip shift sliding block code} induced by block map $\Psi$.
		\end{definition}
		\begin{remark}\label{rem1}
			\begin{itemize}
				\item In Theorem \ref{th 2.12}, we prove that the image of a zip shift sliding block code is in fact a zip shift space. Therefore, we used the notation of the zip shift space $\Sigma_{\mathcal{C},\mathcal{C}'}$ for $\psi(\Sigma_{\mathcal{A},\mathcal{A}'})$.
				\item The equations \eqref{star} and \eqref{starstar} are the result of the following diagram, where $\Psi$ should induce a ``local commutation" rule between the maps $\varphi_n^{\mathcal{A},\mathcal{A}'}$, $\varphi_m^{\mathcal{C},\mathcal{C}'}$, $\sigma_{\Sigma_{\mathcal{A},\mathcal{A}'}}$ and $\sigma_{\Sigma_{\mathcal{C},\mathcal{C}'}}$ as 
				\begin{equation}
					%\begin{center}
					\begin{CD}\label{diagram1}
						[x_{-1}; x_0, x_1 , \cdots , x_{k}]     @> >>  [\varphi_n^{\mathcal{A},\mathcal{A}'}{[x_0, \cdots , x_{n-1}]}; x_1, \cdots , x_{k +1}]\\
						@VV V        @VV V\\
						[y_{-1}; y_0, \cdots , y_{k-n}]     @> >>  [y'_{-1}; y_1, \cdots , y_{k-n+1}].
					\end{CD}
				\end{equation}
				Here, $k=n+m-2$.
				\item The zip shift sliding block code can be defined through a block map $\Theta: \beta^{\mathcal{A},\mathcal{A}'}_{k+2}\rightarrow \beta^{\mathcal{C},\mathcal{C}'}_{k-n+2}$ over blocks of length $k\geq n+m-2$, but still any entry of $\theta (x_{[-1,\, k]})$ is determined by $\Psi$ defined in equations \eqref{star} and \eqref{starstar}. 
				\item By equation \eqref{star}, the definition of  $\Psi$ naturally extends  to blockes with all positive indeces $i\geq 0$. For negative indeces, let  $\ell=\max\{n,m\}$. Then $\ell$ iteration of diagram \eqref{diagram1} defines a rule to convert a block in $\beta^{\mathcal{A}}_{\ell}$ to a block in $\beta^{\mathcal{C}}_{\ell}$.
			\end{itemize}
		\end{remark}
	\begin{theorem}\label{comu}
		Let $\Sigma_{\mathcal{A},\mathcal{A}'}$ and $\Sigma_{\mathcal{C},\mathcal{C}'}$ be two 
		zip shift spaces with transition maps $\varphi_n^{\mathcal{A},\mathcal{A}'}$ and $\varphi_m^{\mathcal{C},\mathcal{C}'}$ as in Definition \ref{def slide}
		and  $\psi:\Sigma_{\mathcal{A},\mathcal{A}'}\rightarrow \Sigma_{\mathcal{C},\mathcal{C}'}$ be a zip shift sliding block code. 
		Let $\sigma_{\Sigma_{\mathcal{A},\mathcal{A}'}}$ and $\sigma_{\Sigma_{\mathcal{C},\mathcal{C}'}}$ be the zip shift maps on spaces 
		$\Sigma_{\mathcal{A},\mathcal{A}'}$ and  $\Sigma_{\mathcal{C},\mathcal{C}'}$  respectively.
		Then $\psi \circ \sigma_{\Sigma_{\mathcal{A},\mathcal{A}'}}= \sigma_{\Sigma_{\mathcal{C},\mathcal{C}'}} \circ \psi$.
	\end{theorem}
	\begin{proof}
		Let $x=(\ldots x_{-1}; x_0,\, x_1, \ldots) \in \Sigma_{\mathcal{A},\mathcal{A}'}$  and  $y=(\ldots y_{-1}; y_0,\, y_1, \ldots) \in \Sigma_{\mathcal{C},\mathcal{C}'}$. Suppose $\psi(x)=y$ according to definition of 
		$\Psi$ 
		in \eqref{star} and \eqref{starstar}. For a word $x_{[-1 ,\,n-1]}\in  \beta^{\mathcal{A},\mathcal{A}'}_{n+1}$, a word $x_{[i ,\,i+n-1]}\in  \beta^{\mathcal{A},\mathcal{A}'}_{n}$ with $i\geq 0$ or  a word $x_{[i ,\,i+\ell-1]}\in  \beta^{\mathcal{A},\mathcal{A}'}_{\ell}$ with $i\leq -\ell$,
		\begin{equation*} %\label{V}
			\begin{aligned}
				\left((\sigma_{\Sigma_{\mathcal{C},\mathcal{C}'}} \circ \psi)(x)\right)_{[i]} 
				&= \Psi \left(\sigma_{\Sigma_{\mathcal{A},\mathcal{A}'}}(x)_{[i,\, i+n-1]} \right) \\
				&= \Psi(x_{[i+1,\, i+n]}) \\
				&= (\psi \circ \sigma_{\Sigma_{\mathcal{A},\mathcal{A}'}}(x))_{[i]}. \\
				&= y_{[i+1]} \\
				&= (\sigma_{\Sigma_{\mathcal{C},\mathcal{C}'}} \circ \psi)_{[i]}.
			\end{aligned}
		\end{equation*}
		For a block $x_{[i ,\,i+\ell-1]}\in  \beta^{\mathcal{A},\mathcal{A}'}_{\ell}$, where $-\ell < i < -1$,  at most $\ell$ iteration of $x$ and $y$ will determine the  entries with negative indeces. Then the above disscusion finishes the proof. 
	\end{proof}

If $\psi:\Sigma_{\mathcal{A},\mathcal{A}'}\rightarrow \Sigma_{\mathcal{C},\mathcal{C}'}$ is onto, then call it a \textit{zip shift factor code} from 
$\Sigma_{\mathcal{A},\mathcal{A}'}$ onto $\Sigma_{\mathcal{C},\mathcal{C}'}$. If $\psi$ is invertible, then it is called a \textit{conjugacy}
from $\Sigma_{\mathcal{A},\mathcal{A}'}$ to $\Sigma_{\mathcal{C},\mathcal{C}'}$. 
Both maps $h_N(x)$ and $\gamma_{N}(x)$ are conjugacies between the space $\Sigma_{\mathcal{A},\mathcal{A}'}$ and 
spaces $h_N(\Sigma_{\mathcal{A},\mathcal{A}'})$
and $\gamma_N(\Sigma_{\mathcal{A},\mathcal{A}'})$ respectively. 
\begin{corollary}
	Let $\psi:\Sigma_{\mathcal{A},\mathcal{A}'}\rightarrow \Sigma_{\mathcal{C},\mathcal{C}'}$ be a sliding block code. Then if $p\in \Sigma_{\mathcal{A},\mathcal{A}'}$
	is a periodic point with period $k$, then the least period of $\psi(p)$ divides $k$. If $\psi$ is a conjugacy, then $p$ and $\psi(p)$ have the same period.
\end{corollary}
The following theorem is a generalization of Curtis-Hedlund-Lyndon Theorem to zip shift maps.
	\begin{theorem}\label{ioi}
		Let $\Sigma_{\mathcal{A},\mathcal{A}'}$  and $\Sigma_{\mathcal{C},\mathcal{C}'}$ be two 
		zip shift spaces with transition maps $\varphi_n^{\mathcal{A},\mathcal{A}'}:\beta^{\mathcal{A}'}_n\longrightarrow \mathcal{A}$ 
		and $\varphi_m^{\mathcal{C},\mathcal{C}'}:\beta^{\mathcal{C}'}_m\longrightarrow \mathcal{C}$ respectively. 
		A map $\psi:\Sigma_{\mathcal{A},\mathcal{A}'}\rightarrow \Sigma_{\mathcal{C},\mathcal{C}'}$ is a zip shift sliding block code 
		if and only if $\psi \circ \sigma_{\Sigma_{\mathcal{A},\mathcal{A}'}}= \sigma_{\Sigma_{\mathcal{C},\mathcal{C}'}} \circ \psi$ 
		and $\psi(x)_{-1}$ is a function of $x_{[-1, \,n  -1]}$ and $\psi(x)_{0}$ is a function of $x_{[0, \,n  -1]}$ such that the block map $\Psi$ induced by $\psi$ satisfies \eqref{starstar}.
	\end{theorem}
	\begin{proof}
		We only need to prove the reverse direction. Let $\psi$ induce a block map $\Psi$ such that it is determined on blocks 
		$x_{[-1, \,n  -1]}$ and $x_{[0, \,n  -1]}$. Since it is defined on blocks $x_{[0, \,n  -1]}$, the definition of $\Psi$ naturally extends to blocks $y_{i}=\Psi(x_{[i,\, i+n-1]})$, $i\geq 0$. 
		Now according to \eqref{starstar} and Remark \ref{rem1}, the definiton extends to blocks of length $n+m+1$ containing negative indices. 
	\end{proof}
\begin{theorem}
	Let $\psi:\Sigma_{\mathcal{A},\mathcal{A}'}\rightarrow \Sigma_{\mathcal{C},\mathcal{C}'}$ be a zip shift sliding block code
	with the  block map $\Psi: \beta^{\mathcal{A},\mathcal{A}'}_{\ell}\rightarrow \mathcal{C}\cup \mathcal{C}'$.
	Then there exists a higher block zip shift space $\widetilde{\Sigma_{\mathcal{A},\mathcal{A}'}}$ of $\Sigma_{\mathcal{A},\mathcal{A}'}$, a conjugacy 
	$h: \Sigma_{\mathcal{A},\mathcal{A}'}\rightarrow \widetilde{\Sigma_{\mathcal{A},\mathcal{A}'}}$ and
	a 1-block code $\widetilde{\psi}:\widetilde{\Sigma_{\mathcal{A},\mathcal{A}'}}\rightarrow \Sigma_{\mathcal{C},\mathcal{C}'}$ such that 
	$\widetilde{\psi} \circ h=\psi$.
\end{theorem}
\begin{proof}
	Take $\widetilde{\Sigma_{\mathcal{A}, \mathcal{A}'}}$ as $\Sigma_{\mathcal{A},\mathcal{A}'}^{[\ell]}$ and define $h$ to be 
	the $\ell$th higher block code $(h(x))_{[i]}=x_{[i,\, i+\ell-1]}$. The map $h$ is a conjugacy and so $\widetilde{\psi}=\psi \circ h^{-1}$ is a 1-block map.
\end{proof}
\begin{theorem}\label{th 2.12}
	\begin{itemize} 
		\item[i)] The image of a zip shift space under the zip shift sliding block code is a zip shift space.
		\item[ii)] A one-to-one and onto zip shift sliding block code has a zip shift sliding block code inverse.
	\end{itemize}
\end{theorem} 
	\begin{proof}
		The proofs of parts i and ii are similar to the proofs of Theorems 1.5.13 and 1.5.14 \cite{LM} respectively.
		
		(i). Let $\Sigma_{\mathcal{A},\mathcal{A}'}$ and $\Sigma_{\mathcal{C},\mathcal{C}'}$ be two 
		zip shift spaces with block map  $\Psi: \beta^{\mathcal{A},\mathcal{A}'}_{\ell}\rightarrow \mathcal{C}\cup \mathcal{C}'$, and $\psi:\Sigma_{\mathcal{A},\mathcal{A}'}\rightarrow \Sigma_{\mathcal{C},\mathcal{C}'}$ be a zip shift sliding block code. 
		According to Definition \ref{def slide} and  Remark \ref{rem1},
		for  $x\in \Sigma_{\mathcal{A},\mathcal{A}'}$,  any block in $\psi(x)$ belongs to  $\beta^{\mathcal{C},\mathcal{C}'}$. Therefore, $\psi(\Sigma_{\mathcal{A},\mathcal{A}'}) \subseteq \Sigma_{\mathcal{C},\mathcal{C}'}$.
		
		To prove the inverse, Observe that according to Remark \ref{rem1} and Theorem \ref{ioi}, the definition of block map $\Psi$ can extend to map centeral blocks $x_{[-k,k]}$ where $k\geq n+m-2$. Now for any $k$ the centeral $(2k-2n+1)$-block of $y$ is the $\Psi$-image of the   centeral $(2k+1)$-block of some point $x^{(k)}\in \Sigma_{\mathcal{A},\mathcal{A}'}$. Since $\mathcal{C}\cup \mathcal{C}'$ contains finite symbols, therefore, there exixts an infinite set $S_k$ of integers such that $x^{(n)}_{[-k,k]}$ is the same for all $n\in S_k$.
		Continiouing this procedure and increasing $k$, we find nested subsets of infinite sets as $\cdots \subseteq S_{k+1}\subseteq S_{k}$. Define $x_{[-k,k]}=x^{(n)}_{[-k,k]}$ for all $n\in S_k$. Therefore, for any $k\geq n+m-2$ we have
		\begin{equation*}
			\Psi(x_{[-k,k]})=\Psi (x^{(n)}_{[-k,k]})= \psi (x^{(n)})_{[-k,k]}=y_{[-k,k]}.
		\end{equation*}
		This means $\psi(x)=y$ and $\Sigma_{\mathcal{C},\mathcal{C}'} \subseteq \psi(\Sigma_{\mathcal{A},\mathcal{A}'})$.
		
		(ii) Let $\psi:\Sigma_{\mathcal{A},\mathcal{A}'}\rightarrow \Sigma_{\mathcal{C},\mathcal{C}'}$ be a one-to-one and onto zip shift sliding block code and $\xi: \Sigma_{\mathcal{C},\mathcal{C}'}\rightarrow \Sigma_{\mathcal{A},\mathcal{A}'}$ be the inverse of $\psi$. From $\psi \circ \sigma_{\Sigma_{\mathcal{A},\mathcal{A}'}}= \sigma_{\Sigma_{\mathcal{C},\mathcal{C}'}} \circ \psi$ and $y=\psi(x)$, we have
		\begin{equation*}
			\sigma_{\Sigma_{\mathcal{A},\mathcal{A}'}}(\xi (y))= \sigma_{\Sigma_{\mathcal{A},\mathcal{A}'}}(x)= \xi(\psi (\sigma_{\Sigma_{\mathcal{A},\mathcal{A}'}})) = \xi (\sigma_{\Sigma_{\mathcal{C},\mathcal{C}'}}(\psi(x)))= \xi (\sigma_{\Sigma_{\mathcal{C},\mathcal{C}'}}(y)).
		\end{equation*}
		Therefore, $\xi \circ \sigma_{\Sigma_{\mathcal{C},\mathcal{C}'}}= \sigma_{\Sigma_{\mathcal{A},\mathcal{A}'}} \circ \psi$.
		Now according to Theorem \ref{ioi}, it suffices to determine $\xi(y)_{[-1]}$ and $\xi(y)_{[0]}$ over blocks $y_{[-1, m-1]}$ and $y_{[0, m-1]}$ respectively. 
		
		Observe that if there exists a block map $\Xi$ defined over blocks $\xi(y)_{[0]}$ and $\xi(y)_{[-1]}$, 
		then its definition would  extend over any block $\xi(y)_{[-t , t]}$, $t\geq m-1$. We prove the theorem for index 0 and the proof for  index -1 is the same. By contradiction suppose that for any $t\geq 0$, there are two points
		$y^{(t)}$ and $\tilde{y}^{(t)}$ in $\Sigma_{\mathcal{C},\mathcal{C}'}$ such that 
		$y^{(t)}_{[-t, t]}=\tilde{y}^{(t)}_{[-t, t]}$, but $\xi (y^{(t)})_{[0]}\neq  \xi (\tilde{y}^{(t)} )_{[0]}$.
		Let $\xi (y^{(t)})= x^{(t)}$ and 
		$\xi (\tilde{y}^{(t)})= \tilde{x}^{(t)}$. Since $\mathcal{A}\cup \mathcal{A}'$ is finite, there would be
		distinct symbols $a\neq b$ and an infinite set $S_0$ of integers so that $ x^{(t)}_{[0]}=\xi (y^{(t)})_{0}=a$ and
		$\tilde{x}^{(t)}_{[0]}=\xi (\tilde{y}^{(t)} )_{0}=b$ for all $t\in S_0$. Then choose an infinite subset 
		$S_1\subset S_0$ so that  $x^{(t)}_{[-1,1]}$ are equal for $t\in S_1$ and $\tilde{x}^{(t)}_{[-1,1]}$ are equal for
		$t\in S_1$.
		Continuing this way, for each $k\geq 1$, we would find  infinite subset $S_k\subset S_{k+1}$.
		
		As in the proof of part (i), we construct points $x$ and $\tilde{x}$ in $\Sigma_{\mathcal{A},\mathcal{A}'}$ such that 
		$x_{[-t,t]}=x^{(k)}_{[-t,t]}$ and $\tilde{x}_{[-t,t]}=\tilde{x}^{(k)}_{[-t,t]}$ for all $n\in S_k$. Note that $x_{[0]}=a\neq b= \tilde{x}_{[0]}$. Therefore, $x\neq \tilde{x}$. Now for $t\in S_k$ and $t\geq k$. we have
		\begin{equation*} %\label{V}
			\begin{aligned}
				\Psi (x_{[-k,k]}) &= \Psi (x^{(t)}_{[-k,k]}) = \psi (x^{(t)})_{[-k,k]} = y^{(t)}_{[-k,k]}\\
				&= \tilde{y}^{(t)}_{[-k,k]} =\psi (\tilde{x}^{(t)})_{[-k,k]} = \Psi (\tilde{x}^{(t)}_{[-k,k]}) = \Psi (\tilde{x}_{[-k,k]}).
			\end{aligned}
		\end{equation*}
		This would imply that $\psi (x) =\psi (\tilde{x})$ which is a contradiction.
		
		If condition \eqref{starstar} does not hold for $\Xi$, it contradicts the map $\xi$ to be one-to-one.
	\end{proof}

\begin{remark}
	\begin{itemize}
		\item
		Let $\psi:\Sigma_{\mathcal{A},\mathcal{A}'}\rightarrow \Sigma_{\mathcal{C},\mathcal{C}'}$ be a one-to-one sliding block code induced by $\Psi$. Note that the map $\Psi$ does not need to be one-to-one. %For example ... ??
		\item For $\ell\geq n$ and $\ell'\geq m$, the block map $\Psi$ can be defined as 
		$$\Psi: \beta_{\ell}^{\mathcal{A},\mathcal{A}'}\rightarrow \beta_{\ell'}^{\mathcal{C},\mathcal{C}'}$$
		with $\Psi(x_{[i,\, i+\ell-1]})=y_{[i,\, i+\ell'-1]}.$ According to the indices, the letters in $y_{[i,\,i+\ell'-1]}$ are chosen from $C\cup C'$.
	\end{itemize}
\end{remark}

\section{Pre-images of points under zip shift map}\label{sec3pre}
Let $\Sigma\subseteq \Sigma_{\mathcal{A},\mathcal{A}'}$ be an SFT or strictly sofic space and 
$\sigma : \Sigma \rightarrow \Sigma$ be a zip shift map over  $\Sigma$. 
For a point $x\in \Sigma$, we aim to find the set of  all pre-images of it under $\sigma$. To do this,  
we construct a labeled graph called ``backward labeled graph'' which shows the backward dynamics of the points in $\Sigma_{\mathcal{A},\mathcal{A}'}$.

Let $\varphi_n: \beta_n^{\mathcal{A}'} \rightarrow \mathcal{A}$ be  the transition map. If $a_i\in \mathcal{A}$,   $a_i'\in \mathcal{A}'$ and
\begin{equation}\label{x}
	x=(\ldots, a_{-2},\, a_{-1};\, a_0',\, a_1',\, a_2',\, \ldots )\in \Sigma,
\end{equation}
then
$$\sigma (x)=(\ldots, a_{-2},\, a_{-1},\, a;\, a_1',\, a_2',\, \ldots ),$$
where  $a=\varphi_n(a_0'\, a_1'\, \ldots a_{n-1}')\in \mathcal{A}$.
We let 
\begin{equation}\label{-1}
	\sigma^{-1}(x)=(\ldots, a_{-2};\, a',\, a_0',\, a_1',\, a_2',\, \ldots ),
\end{equation}
where  $a'$ is an ``appropriate" pre-image of  $a_{-1}\, a_0'\, a_1' \ldots a_{n-2}'$  under $\varphi_n$. Sometimes, deciding which pre-image is appropriate is not easy at all. 
\begin{example}\label{eg p }
	Consider the simple graph in Figure \ref{Fpre-image} with $\mathcal{A}'=\{1,2,3,4,5,6\}$ and $\mathcal{A}=\{a,b,c,d,e\}$. 
	Set $\varphi_1(1)=a$, $\varphi_1(2)=\varphi_1(3)=b$, $\varphi_1(4)=c$, $\varphi_1(5)=d$ and $\varphi_1(6)=e$.
	The space $\Sigma_{\mathcal{A},\mathcal{A}'}$ is sofic. Let
	$$x=(\ldots, b\, ,b\, ,b;\, 1,\, 6,\ldots)\in \Sigma_{\mathcal{A},\mathcal{A}'}.$$
	Since $\varphi_{1}^{-1}(b)=\{2,\,3\}$, we cannot decide which one to choose in order to determine $\sigma^{-1}(x)$.
	In fact, we should trace back the sequece of $b$'s in negative indices up to an entry $x_{-k}$ which is $c$ or $e$. 
	If $x_{-k}=c$, then $y=\sigma^{-1}(x)=(\ldots ,b\, ,b;\, 2,\,1,\, 6,\ldots)$, and if  $x_{-k}=e$, 
	then $z=\sigma^{-1}(x)=(\ldots  ,b\, ,b;\, 3,\,1,\, 6,\ldots)$. If there does not exist $k$ with $x_{-k}=c$ or $e$, then both points $y$ and $z$ are pre-images of $x$.
\end{example}
\begin{figure}
%	
%	\begin{tikzpicture}
%		\node[circle,draw] (A) at (0,0) {};
%		\node[circle,draw] (B) at (3,0) {};
%		\node[circle,draw] (D) at (0,2) {};
%		\node[circle,draw] (C) at (3,2) {};
%		
%		\path [style={->, >=latex', shorten >=1pt}] (A) edge [ loop left] node[left ] {$3/b$} (A);
%		\path [style={->, >=latex', shorten >=1pt}] (A) edge [ ] node[left ] {$3/b$} (D);
%		\path [style={->, >=latex', shorten >=1pt}] (B) edge [ ] node[below] {$5/d$} (A);
%		\path [style={->, >=latex', shorten >=1pt}] (B) edge [ ] node[right] {$4/c$} (C); 
%		\path [style={->, >=latex', shorten >=1pt}] (C) edge [ loop right] node[right ] {$2/b$} (C);
%		\path [style={->, >=latex', shorten >=1pt}] (C) edge [ ] node[above] {$2/b$} (D);
%		\path [style={->, >=latex', shorten >=1pt}] (D) edge [ loop left] node[left ] {$1/a$} (D);
%		\path [style={->, >=latex', shorten >=1pt}] (D) edge [ ] node[left=0.2] {$6/e$} (B);
%		
%	\end{tikzpicture}
\includegraphics[width=0.45\textwidth]{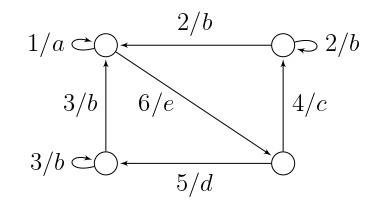}
\caption{This graph represents a zip shift space $\Sigma_{\mathcal{A},\mathcal{A}'}$ with $\mathcal{A}'=\{1,2,3,4,5,6\}$ and $\mathcal{A}=\{a,b,c,d,e\}$. 
		Walking in direction of edges and considering the numbers on edges, we get entries with non-negative indices and walking in opposite direction of edges and picking
		the letters written on edges, we get the entries with negative indices of a point $x\in \Sigma_{\mathcal{A},\mathcal{A}'}$.}
	\label{Fpre-image}
\end{figure}

In sequel we want to find the pre-images of a given point $x\in \Sigma_{\mathcal{A},\mathcal{A}'}$ through the backward labeled graph associated to
$\Sigma_{\mathcal{A},\mathcal{A}'}$.
We consider two cases when $\Sigma_{\mathcal{A}'}$ is an $M$-step Markov space or when $\Sigma_{\mathcal{A}'}$ is a strictly sofic space.	
\subsection{Constructing the labeled  and backward labeled graphs for sofic zip shift spaces}\label{sect}

In \cite{LM},  $M$-step Markov spaces and sofic spaces are represented by  graphs. 
A vertex graph or an edge  graph associates to the first space and a  labeled graph represents the second one. 
Here, we correspond a  graph to the sofic zip shift space which carries labels both on its vertices and edges. Call it \textit{labeled graph} for zip shift space.
The  \textit{backward labeled graph} can be obtained from the labeled graph by reversing the direction of edges. 
It is easy to construct the labeled graph when $\Sigma_{\mathcal{A}'}$ is an $M$-step Markov space. If $\Sigma_{\mathcal{A}'}$ is strictly sofic, some modifications will be made. 

$\mathbf{Case\ 1}.$ Let $\Sigma_{\mathcal{A}'}$ be an $M$-step Markov space. If $n=M=1$ ($n$ is due to map $\varphi_n$), then take the edge graph representing 
$\Sigma_{\mathcal{A}'}$. Here, the labels on the edges are taken from $\mathcal{A}'$. If the label of an edge is $a'$, then replace it with $a' / \varphi_1(a')$.

If $n=1$, set $\mathcal{V}(H)=\beta_M^{\mathcal{A}'}$ as the vertex set of the labeled graph. 
There is a directed edge $e$ from $v_1=a_1' a_2'\ldots a_{M}'$ to $v_2=a_2'\ldots a_{M}'a_{M+1}'$ if $a_1'a_2'\ldots a_{M}'a_{M+1}'\in \beta_{M+1}^{\mathcal{A}'}$. 
Label this edge by 
$$a_{M+1}' / \varphi_1(a_1').$$
If there are $m$ letters $b_1'$, $\ldots$, $b_m'\in \mathcal{A}'$ such that $b_j'a_2'a_3'\ldots a_{M+1}'\in \beta_{M+1}^{\mathcal{A}'}$, $1\leq j\leq m$, 
then $m$ edges  enters to the vertex $a_2'a_3'\ldots a_{M+1}'$ with labels $a_{M+1}'/ \varphi_1(b_j')$.
Note that the letters $b_j'$ are different, however the sofic parts $\varphi_1(b_j')$ may be the same for some edges.

If $n\geq 2$, let $\ell=\max \{n-1,\, M\}$.  Set $\mathcal{V}(\mathcal{H})=\beta_{\ell}^{\mathcal{A}'}$. 
There exists an edge $e$ from $v_1=a_1'\ldots a_{\ell}'$ to $v_2=a_2'a_3'\ldots a_{\ell+1}'$ 
if $a_1'a_2'\ldots a_{\ell}'a_{\ell+1}'\in \beta_{\ell+1}^{\mathcal{A}'}$.  Label this edge by 
$a_{\ell+1}' / \varphi_n (a_1'\ldots a_{n}')$.

$\mathbf{Case\ 2}.$ Let $\Sigma_{\mathcal{A}'}$ be a strictly sofic shift space. 
Call a graph a \textit{vertex-like graph} if it is a vertex graph such that some of the vertices carry the same label. See Figure \ref{Fsofic}.
\begin{lemma}\label{lem}
	Let $X$ be a sofic shift space.  
	There exists a directed vertex-like graph $H$ which  represents $X$.
\end{lemma}
\begin{proof}
	Let ${G}$ be a labeled graph representing $X$. Let $e$ be an edge from vertex $v'$ to $w'$.  Denote by $i(e)$ and $t(e)$ the initial and terminal vertices of $e$, 
	i.e., $i(e)=v'$ and $t(e)=w'$.  The graph ${G}$ might have at least two edges $e_1$ and $e_2$ with $i(e_1)=i(e_2)=I$ and $t(e_1)=t(e_2)=J$.
	As described in Section 2.4, \cite{LM}, we can split $I$ into two vertices $I_1$ and $I_2$ such that $i(e_1)=I_1$ and $i(e_2)=I_2$.
	For all such vertices we split initial vertices into two or more vertices until all of the edges in the new graph ${G}'$ have different initial vertices o different terminal vertices.  Then $X_{{G}}=X_{{G}'}$ \cite{LM}.
	The graph  ${G}'$  contains some edges with the same label. If ${G}'$ has $k$ such  edges, rename all edges with $e_1$ to $e_k$. 
	There exists a vertex graph conjugate  to ${G}'$ with vertices $e_1$ to $e_k$. Again rename each $e_i$ by the original
	label as in ${G}'$. This vertex-like graph ${H}$ has some vertices with the same label and represents $X$.
\end{proof}

For  $n\geq 1$, let $\mathcal{H}'$ be the vertex-like graph with vertex set $\mathcal{V}(\mathcal{H}')=\beta_{\ell}^{\mathcal{A}'}$, $\ell\geq n-1$ 
representing $\Sigma_{\mathcal{A}'}$ (for $n=1$, take $\ell\geq 1$ and).  For $n=1$, one can also consider the labeled graph of $\Sigma_{\mathcal{A}'}$.
Suppose $v_1$ and $v_2$ are two vertices in $\mathcal{G}'$ and $e$ is an edge from $v_1$ to $v_2$. 
Let $v_2$ end with letter $a'\in\mathcal{A}'$. 
Then label $e$ by $a'\, / \varphi_n(w_{[1,n]})$ where $w=v_1a'$.

\begin{example}\label{sofic eg}
	Let $\mathcal{A}'=\{ 0,1\}$. The space $\Sigma_{\mathcal{A}'}$ with forbidden block set $\mathcal{F}=\{ 10^{2n-1}1:\ \ n\in \mathbb{N}\}$ is called even shift. 
	Its labeled graph is shown in Figure \ref{Fsofic} parts a and b. 
	There is a condition  on the number of symbols zero between two successive 1's. 
	To show this space with a vertex-like graph, it suffices to duplicate the vertex with label zero. The new graph  illustrated in Figure \ref{Fsofic}c, is a vertex-like graph for the even shift.
	
	If the vertex set is $\beta_2^{\mathcal{A}'}$, then another representation of the even shift by a vertex-like graph is as in Figure \ref{Fsofic}e. Let $\mathcal{A}=\{a,\,b,\,c\}$ and 
	$\varphi_2(11)=a$, $\varphi_2(10)=\varphi_2(01)=b$, and $\varphi_2(00)=c$. 
	Its  labeled graphs are shown in Figure \ref{Fsofic}, parts d and f.
\end{example}
\begin{figure}
	\centering
	\includegraphics[width=0.8\textwidth]{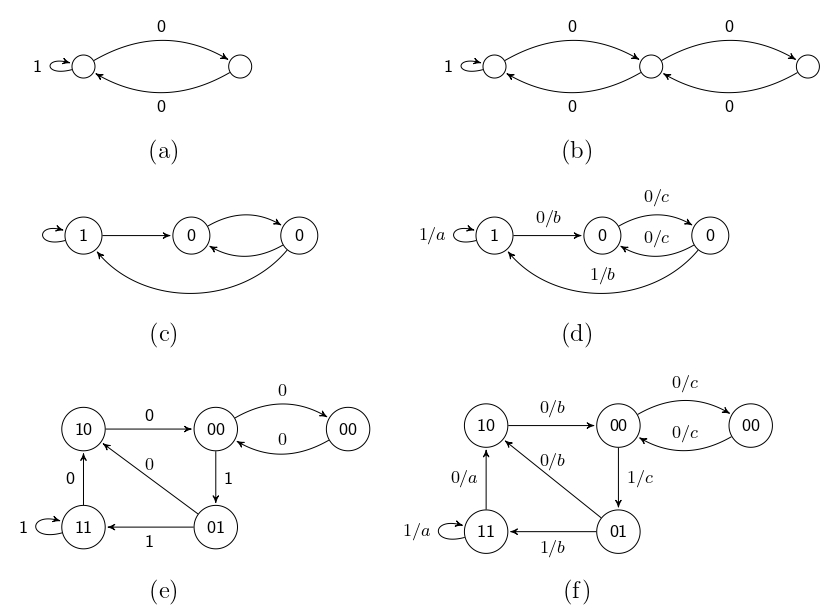}
	\caption{Graphs (a), (b), (c), and (e) represent the labeled or vertex-like graphs for the even shift $\Sigma_{\mathcal{A}'}$. Graphs (d) and (f) represent the labeled graphs for $\Sigma_{\mathcal{A},\mathcal{A}'}$ with the transition map described in Example \ref{sofic eg}.}
	\label{Fsofic}
\end{figure}	
\subsection{Walking on graphs associated to a sofic space $\Sigma_{\mathcal{A},\,\mathcal{A}'}$}\label{sec:3.2}
For a  sofic space $\Sigma_{\mathcal{A},\,\mathcal{A}'}$ with  transition map $\varphi_n: \beta_n^{\mathcal{A}'} \rightarrow \mathcal{A}$,  
let $\mathcal{H}$ and $\mathcal{G}$ be a labeled and a backward labeled graph with vertex set 
$\beta_{\ell}^{\mathcal{A}'}$ for $\ell\geq n-1$ respectively.
Let $e_j$ connect the vertex $v_j$ to $v_{j+1}$, carrying the label $a_j' / a_j$. 
For a label $a_j' / a_j$ with $a_j'\in \mathcal{A}'$ and $a_j\in \mathcal{A}$, we call $a_j'$ and $a_j$  the \textit{first} and the \textit{sofic part} of the label. 
Walks or paths on these graphs can have finite or infinite length.

First consider the graph $\mathcal{H}$.
Any finite segment $\pi=e_i\cdots e_{i+m}$ of a path in the direction of edges  is a concatenation  of edges
$e_j$ with $t(e_j)=i(e_{j+1})$ for $0\leq i\leq j\leq i+m-1$. This path corresponds to a block $a'_i\ldots a'_{i+m}\in \beta^{\mathcal{A}'}_{m+1}$ 
where $a'_i$ is the first label of $e_i$.

The graph $\mathcal{G}$ is constructed  from $\mathcal{H}$ by  reversing  the direction of edges without changing the lanels of edges.
Walking in opposite direction of edges in $\mathcal{H}$ or in the direction of edges in $\mathcal{G}$, any finite 
segment $\pi=e_{i+m}\cdots e_{i}$ of a path  is a concatenation  of edges
$e_j$ with $t(e_j)=i(e_{j+1})$ for $i,\, m>0$ and $i\leq j\leq i+m-1$.
This path starts from $e_i$ and ends at $e_{i+m}$.
The indices are sorted this way to show their dependency 
to the entries $(\ldots, a_{-(i+m)},\ldots, a_{-(i+1)},\,a_{-i},\ldots ;\ldots)$.   
This path corresponds to a word $a_{-(i+m)}\ldots a_{-i}\in \beta^{\mathcal{A}}_{m+1}$ 
where $a_{-i}$ is the sofic label of $e_i$.
The block  $L:=a_{-(i+m)}\ldots a_{-i}$ is called the \textit{sofic label} of path $\pi$ and is denoted by $\mathcal{L}(\pi)$.  
Denote by $i(\pi)$ and $t(\pi)$ the initial and terminal vertices of $\pi$. For both above paths $\pi$, $i(\pi)=e_i$ and $t(\pi)=e_{i+m}$.
	
	\begin{definition}\label{2.4}
		Let $\mathcal{G}$ be a labeled  backward  graph for the sofic space $\Sigma_{\mathcal{A},\mathcal{A}'}$.
		\begin{itemize}
			\item Let $v'$ be a vertex of $\mathcal{G}$ and $L$  be an admissible sofic label. 
			A path $\pi=e_{d}e_{d-1}\ldots e_{1}$ emanating  from $v'$ has \textit{delay} $d$ if $d$ is the smallest natural number  
			and $\pi$ is  the  unique path such that $\mathcal{L}(\pi)=L$.
			\item Let $v'$ be a vertex of $\mathcal{G}$ and let $L$  be an admissible sofic label of  paths $\pi_1, \ldots, \pi_m$, $m\geq 2$  all of length
			$\ell$ with  $i(\pi_1)=\ldots =i(\pi_m)=v'$ and $t(\pi_1)=\ldots =t(\pi_m)$. Call these paths  \textit{left-closing} paths of length $\ell$ if $\ell$ is 
			the smallest natural number and $\pi_j$'s are the only paths emanating from $v'$ having sofic label $L$ (see Figures \ref{l1} and \ref{Fclosing}).
			\item  
			A sofic label $L$ is called \textit{$d_L$-distinguishable} if the number $d$ and the label $L$ determines a 
			unique vertex $v'$ and a unique path $\pi$ such that $i(\pi)=v'$ and $\pi$ has delay $d$. 
			\item A sofic label $L$ is called \textit{$L$-distinguishable} if there exist a unique vertex $v'$ and some paths $\xi_j=\ldots e_3^j e_2^j e_1^j$, $1\leq j\leq t$ where $i(\xi_j)=v'$ and
			$\mathcal{L}(\xi_j)=L$.
			\item Suppose there are vertices $v_1'$ to $v_s'$, $s\geq 2$ with label $a_0'\ldots a_{n-1}'$ and for sofic label $L$,
			each of these vertices is the initial vertex of at least one of the  paths $\pi_j$, $1\leq j\leq m$, $m\geq 2$ carrying the label $L$.
			If all of these paths have the same endpoint $w'$, then they are called \textit{sofic-left-closing}.
		\end{itemize}
	\end{definition}	
Let $\mathcal{G}$ be a  backward labeled graph  for $\Sigma_{\mathcal{A},\mathcal{A}'}$.
To obtain the pre-images of a point $x=(\ldots, a_{-2},\, a_{-1};\, a_0',\, a_1',\,  \ldots )$, first  find the vertex $v'$ or 
vertices $v_1'$ to $v_s'$ which determine the starting state according to the sequence $ (a_0',\, a_1',\, a_2',\, \ldots )$. 
Suppose there are more than one edge with label $a_{-1}$ emanating   from such vertices. It might happen that only some of the paths starting with these edges 
have label $(\ldots, a_{-2},\, a_{-1})$ to form the pre-images of point $x$. In some cases it is very hard to verify if a finite part of $(\ldots, a_{-2},\, a_{-1})$ 
suffices to decide whether the paths carrying such label produce the pre-images of $x$ or not.
\begin{theorem}\label{main2}
	Let $\varphi_n: \beta_n^{\mathcal{A}'} \rightarrow \mathcal{A}$ be an onto factor map. 
	If $\Sigma_{\mathcal{A}'}$ is  a sofic space, then the pre-images of a point 
	$x\in \Sigma_{\mathcal{A},\mathcal{A}'}$ under $\sigma$ can be determined by its  backward labeled graph according to the following rules.
	\begin{itemize}
		\item[i)] If $\Sigma_{\mathcal{A}'}$ is an $M$-step Markov space, and a sofic label $L$ occurs with delay $d$ or is the label of left-closing paths, then all of these paths form
		the pre-images of $x$. Otherwise, all paths with sofic label $\mathcal{L}=\ldots a_{-3}\,a_{-2}\,a_{-1}$ produce the pre-images of $x$.
		\item[ii)]  If $\Sigma_{\mathcal{A}'}$ is a strictly sofic space, and a sofic label $L$ is $d_L$-distinguishable, $L$-distinguishable or is the label of sofic-left-closing paths,
		then all of these paths form    the pre-images of $x$. Otherwise, all paths with sofic label $\mathcal{L}=\ldots a_{-3}\,a_{-2}\,a_{-1}$ produce the pre-images of $x$.
	\end{itemize}
\end{theorem}
\begin{remark}
	The proof of Theorem \ref{main2} can be used to determine the pre-images of points in a one-sided sofic shift space with some adjustments.
\end{remark}
\begin{proof}
	Let 
	$$K_n=\{w'\in \beta_n^{\mathcal{A}'}\, : \exists \, v'\in \beta_n^{\mathcal{A}'}\, :\, \varphi_n(w')=\varphi_n(v')\},
	$$ where $n$ is due to the transition map.
	For $w'=w_1'\ldots w_n'\in K_n$, let
	$$R_n(w')=\{v'=v_1'\ldots v_n'\in K_n \setminus \{w'\}: \ w_2'\ldots w_n'=v_2'\ldots v_n'\}$$
	and $R(n)=\cup_{w\in K_n}R_n(w)$.
	If $\varphi_n$ is one-to-one, then $K_n$ and $R(n)$ are empty sets and $\sigma$ is one-to-one on $\Sigma_{\mathcal{A},\mathcal{A}'}$.
	Suppose $\varphi_n$ is not one-to-one. So, $K_n\neq \emptyset$. Suppose $R(n)=\emptyset$. In this case $\sigma$ is again invertible. 
	To see this let $x$ be as in \eqref{x} and for a letter $a'\in \mathcal{A}'$, $w'=a'a_0'a_1'\ldots a_{n-2}'\in K_n$. 
	Since $R(n)=\emptyset$, only for the word $w'$, we have $\varphi_n(w')=a_{-1}$ and so, $\sigma^{-1}(x)$ is as in \eqref{-1}.
	If $R(n)\neq \emptyset$, then   $\sigma$ might be invertible or not as we discuss in sequel.
	
	First  let $\Sigma_{\mathcal{A}'}$ be an $M$-step Markov space  and $\mathcal{G}$ be its backward labeled graph. Suppose $n=1$. Without loss of generality, suppose for the vertex $v'=a_0'a_1'\ldots a_{M-1}'\in \mathcal{V}(\mathcal{G})=\beta_M^{\mathcal{A}'}$, 
	there are at least two edges $e_1$ and $e_2$ emanating  from $v'$ with $\mathcal{L}(e_1)=\mathcal{L}(e_2)=a_{-1}$. Then one of the following 3 cases  happen.
	\begin{enumerate}
		\item There exists $d\geq 1$ and a path $\pi=e_{d}\ldots e_1$ with sofic label $\mathcal{L}(\pi)=a_{-d}\ldots a_{-1}$, such that $i(\pi)=v'$ and $\pi$ has delay $d$.
		\item There are $m\geq 2$ left-closing paths $\pi_j=e^j_{\ell} e^j_{\ell-1}\ldots e_2^j e_1^j$, $1\leq j\leq m$ of length $\ell$ with $i(\pi_j)=v'$,  $t(\pi_j)=w'$ and $\mathcal{L}(\pi_j)=a_{-d}\ldots a_{-1}$.
		\item  There are infinite length paths $\xi_j=\ldots e_3^j e_2^j e^j_1$, $1\leq j\leq t$  carrying the label $\mathcal{L}(\ldots e_3^j e_2^j e_1^j)=\ldots a_{-3} a_{-2}a_{-1}$, 
		where neither of them happen  with delay $d$ nor is a left-closing path with some length $\ell$. 
	\end{enumerate}
	Any vertex $v'$ admits only one of these cases.
	\begin{itemize}
		\item If Case (1) happens for the vertex $v'$ and a path $\pi=e_{-d}\ldots e_{-1}$  which starts from $v'$ with delay $d$, then $\sigma$ is invertible on $x$ and for $1\leq k\leq d$,
		\begin{equation}\label{pre 1}
			\sigma^{-k}(x)=(\ldots,a_{-k-1};\, a_{-k}',\, a_{-k+1}',\ldots ),
		\end{equation} 
		where $a_{-k-1}'$ is the first part of the label of edge $e_k$. 
		
		\begin{figure}
		\includegraphics[width=0.85\textwidth]{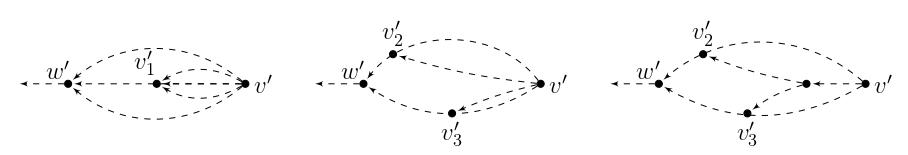}
			\caption{\label{l1}\small  By definition, the paths ending at $w'$, form left-closing paths and the smaller groups of paths ending at $v_1'$, $v_2'$ or $v_3'$ are not left-closing.}
			\label{Floop}
		\end{figure}
		
		\item If $v'$  satisfies   the second case, then $x$ has $m$ pre-images along the paths $\pi_1$ to $\pi_{m}$, i.e.,
		the pre-images of $x$ are the set of points 
		\begin{equation}\label{fpre}
			(\ldots, a_{-2}; (a^j)_{-1}',\, a_0',\ldots ),
		\end{equation}
		where $(a^j)_{-1}'$ is the first part of the label of edge $e^j_1$. Also,   for $1\leq k\leq \ell$ and $1\leq j\leq m$, $\sigma^{-k}(x)$ is the set of points 
		\begin{equation}\label{pre 2}
			(\ldots, a_{-k-1};(a^j)_{-k}',\, (a^j)_{-k+1}',\ldots, (a^j)_{-1}',\,a_0,a_1,\ldots),
		\end{equation}
		where $(a^j)_{-k}'$ is the first part of the label of edge $e^j_k$. In Figure \ref{Floop}, there are some kinds of left-closing paths.
		\item If  the third case happens, then all paths $\pi_j= \ldots e_3^j e_2^j e_1^j$ with $i(\pi_j)=v'$ and $\mathcal{L}(\ldots e_3^j e_2^j e_1^j)=\ldots a_{-3} a_{-2}a_{-1}$ determine the pre-images of $x$. Some cases are shown in Figure \ref{Fclosing}.
		
		\begin{figure}
%			\begin{tikzpicture}
%				\node[] (A) at (0,0) {4};
%				\node[shape=circle] (B) at (0,1) {3};
%				\node[shape=circle] (C) at (0,2.5) {2};
%				\node[shape=circle] (D) at (0,4) {1};
%				\node[circle,draw,inner sep=1.2,fill=black] (E) at (3,2.5) {};
%				\node[circle,draw,inner sep=1.2,fill=black] (F) at (6,2.5) {};
%				\node[circle,draw,inner sep=1.2,fill=black] (G) at (1,0) {};
%				\node[ circle,draw,inner sep=1.2,fill=black] (H) at (2,0) {};
%				\node[circle,draw,inner sep=1.2,fill=black] (I) at (3.5,1) {}; 
%				\node[] (J) at (2,-0.35) {$v_1'$};
%				
%				\path [style={->, >=latex', shorten >=1pt}, bend right, dashed] (F) edge node[ ] {} (D);
%				\path [style={->, >=latex', shorten >=1pt}, bend right, dashed] (F) edge node[] {} (E);
%				\path [style={->, >=latex', shorten >=1pt}, dashed] (F) edge node[] {} (E);
%				\path [style={->, >=latex', shorten >=1pt}, bend left=20, dashed] (F) edge node[ ] {} (E);
%				\path [style={->, >=latex', shorten >=1pt}, bend right=10, dashed] (F) edge node[ ] {} (I);
%				\path [style={->, >=latex', shorten >=1pt}, bend left, dashed] (F) edge node[ ] {} (H);
%				\path [style={->, >=latex', shorten >=1pt}, right, bend right=5, dashed] (I) edge node[ ] {} (H);
%				\path [style={->, >=latex', shorten >=1pt},  left] (H) edge node[above ] {$a$} (G);   
%				\path [style={->, >=latex', shorten >=1pt},  left, dashed] (G) edge node[] {} (A);
%				\path [style={->, >=latex', shorten >=1pt}, dashed] (I) edge node[ ] {} (B);   
%				\path [style={->, >=latex', shorten >=1pt}, dashed] (E) edge node[ ] {} (C);   
%			\end{tikzpicture}
\includegraphics[width=0.45\textwidth]{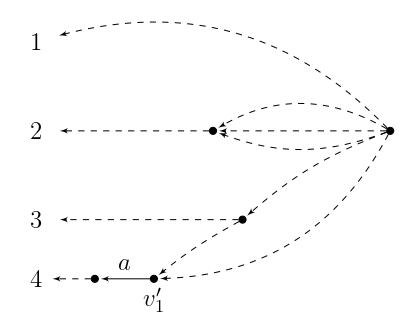}		
\caption{ We take all the paths associated to the numbers (1), (2) and (3) to have sofic label $\ldots a_{-3}\, a_{-2}\, a_{-1}$. For number (4), the sofic label of paths are
				$a_{-k}\,\ldots\,  a_{-3}\, a_{-2},\, a
				_{-1}$ up to vertex $v_1'$, but the next sofic label is ``a" instead of being $a_{-k-1}$. This means that none of the paths in number (4) produces the pre-images of $x$. 
				In this figure there are not any left-closing paths.}\label{Fclosing}
		\end{figure}
		
	\end{itemize}
	For $n>1$, the argument is similar and one  needs to consider the backward labeled graph for case $n>1$.
	Note that the length of the vertices in $\mathcal{G}$ should be greater than $n-1$ and if so, their lengths effect neither on the sets $K_n$ and $R_n$ nor on the Definition \ref{2.4}.
	
	Now suppose $\Sigma_{\mathcal{A}'}$ is a strictly sofic space and $\mathcal{G}$ is its backward labeled graph as described in Subsection
	\ref{sect}. Let $\mathcal{V}(\mathcal{G}')=\beta_n^{\mathcal{A}'}$. Choose the vertex with label $v'=a_0'\ldots a_{n-1}'$.
	If there is only one vertex with this label, then the argument is similar to the Markov case.
	
	Suppose there are vertices $v_1'$ to $v_s'$ with label $a_0'\ldots a_{n-1}'$ such that starting from these vertices and walking
	on $\mathcal{G}$, the first label of the edges is  $a_0'\, a_1'\, a_2'\, \ldots $. 
	\begin{itemize}
		\item If there exists $d$ such that the sofic label  $L=a_{-d}\ldots a_{-1}$ is $d_{L}$-distinguishable or if the
		label $L=\ldots a_{-2}a_{-1}$ is $L$-distinguishable, then there exists a unique path $\pi$ with
		label $L$ and a unique $t$, $1\leq t\leq s$, such that  $i(\pi)=v_{\ell}'$ and the pre-images
		of $x$ satisfy \eqref{pre 1} (In $L$-distinguishable case, $d\rightarrow \infty$).
		Let for $1\leq j\leq k$, $t(\pi)=v_j'$.  Then take $v_j'$ as $v'$ and the argument is as the Markov case.
		\item Suppose for $m\geq 2$ and $1\leq j\leq m$, the paths $\pi_j=e^j_{\ell} e^j_{\ell-1}\ldots e_2^j e_1^j$ are sofic-left-closing, 
		$t(\pi_j)=w'$ and $i(\pi_j)\in \{ v_1',\ldots,v_s'\}$. Then all such paths produce pre-images of $x$ and they satisfy \eqref{fpre} and \eqref{pre 2}. Again take $w'$ as $v'$ and the argument is as the Markov case.
		\item  If  there are paths $\pi_j= \ldots e_3^j e_2^j e_1^j$ with $i(\pi_j)\in \{ v_1',\ldots,v_s'\}$   and 
		$\mathcal{L}(\ldots e_3^j e_2^j e_1^j)$ $ = \ldots a_{-3} a_{-2}a_{-1}$, then all of them  determine the pre-images of $x$.
	\end{itemize}
\end{proof}	

\section{Construction of an  $N$-to-1 endomorphism horseshoe}\label{ho}
Let $\textbf{S}$ be the $[0,\,1]\times [0,\,1]$ square. We aim to construct an $N$-to-1 horseshoe 
map $f$ on a subset of $\textbf{S}$ according to the  following description.
Let $f:\mathbb{R}^2\rightarrow \mathbb{R}^2$ linearly expand $\textbf{S}$ in the horizontal direction by a  factor $\delta$ greater than $2N$, $N\in \mathbb{N}$  and
linearly contract it  in the vertical direction by a factor  $\delta'$ less than $\frac{1}{N}$ such that $\delta\delta'=N$..
If $\delta=2N+\epsilon$,  then  taking modulo by $2+\frac{\epsilon}{N}$ or folding the rectangle successively, $N$ copies of the rectangles 
$[0,\,\frac{\delta}{N}]\times [0,\,\delta']$ overlay.
Now bend these rectangles  back inside $\textbf{S}$.
Obviously, $f$ is $N$-to-1 on $\textbf{S}$. 
Note that  for any $k\geq 2$, $f^k(\textbf{S})$ is  \textbf{one} horizontal ``snake-like" strip lying inside $ f^{k-1}(\textbf{S})$. 
Since $f$ shrinks $\textbf{S}$ in vertical direction by $\delta'$, $\lim_{k\rightarrow +\infty}f^k(\textbf{S})$ tends to an infinite-length horizontal curve $c^u$.

Inside the square, $f(\textbf{S})\cap \textbf{S}$ is consisted of two horizontal rectangles named $H_a$ and $H_b$. 
The map $f$ has the Markov property, i.e., the iterates of $f$ on $\textbf{S}$ generate thin rectangles 
such that $f^k(\textbf{S})\cap \textbf{S}\subset f^{k-1}(\textbf{S})\cap \textbf{S}$, $k\geq 1$. 
\textbf{S}ee Figure \ref{1} for $N=2$ and $k=1,2$.
\begin{figure}[t]
	\includegraphics[width=0.6\textwidth]{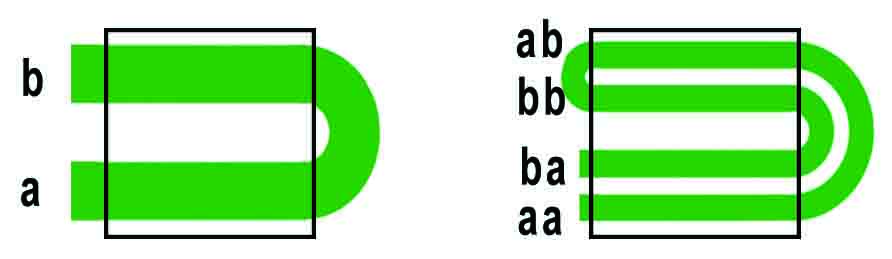}
	\caption{The rectangles $H_a$ and $H_b$ are fromed from  the first iterate of $f$. The rectangles $H_{aa}$, $H_{ab}$, $H_{ba}$ and $H_{bb}$ 
		are shown with labels $aa$, $ab$, $ba$ and $bb$.}
	\label{1}  
\end{figure}
In each iterate, $f$ shrinks the vertical side by factor $\delta'$.
Therefore, a  sequence of nested horizontal rectangles $f^k(\textbf{S})$ tends to a horizontal line segment as $k\rightarrow \infty$.
This  implies that   
$$\Lambda_H=\{(x,\,y)\in \textbf{S}:\ (x,\,y)\in \lim_{k\rightarrow \infty}f^k(\textbf{S})\cap \textbf{S}\}=[0,\,1]\times C_2,$$
where $C_2$ is isomorphic to the Cantor middle third  set. Any point in $\Lambda_H$ remains in $\Lambda_H$ under the iterations of $f$.

The pre-image of $\textbf{S}$ under $f$ is consisted of $N$ thin vertical rectangles.
The place of such rectangles depends on map $f$. In  Figure \ref{2}, since $f$ takes modulo by $2+\frac{\epsilon}{N}$,  the vertical rectangles 
in backward iterate are placed at equal distances. 
Each of these rectangles are the pre-images of $\textbf{S}$ which are linearly stretched by factor $1/\delta'$ and linearly shrank by
factor $1/\delta$ in the vertical and horizontal directions respectively. 
Bending the rectangles back inside the square, $f^{-1}(\textbf{S})$ is consisted of $N$ bended strips. Then $f^{-2}(\textbf{S})$ produces $N$ vertical snake-like strips inside each of
existing strips. This means that $f^{-k}(\textbf{S})$ tends to infinitely many vertical snake-like curves $c^s$ as $k\rightarrow \infty$ all of them lying inside $f^{-1}(\textbf{S})$.

Also,  $f^{-1}(\textbf{S})\cap \textbf{S}$ produces $2N$  
vertical rectangles $V_{0^1}, V_{1^1}, V_{0^2}, V_{1^2},\ldots V_{0^N}, V_{1^N}$.  See Figure \ref{2} for $N=2$. 
\begin{figure}[t]
	\includegraphics[width=0.55\textwidth]{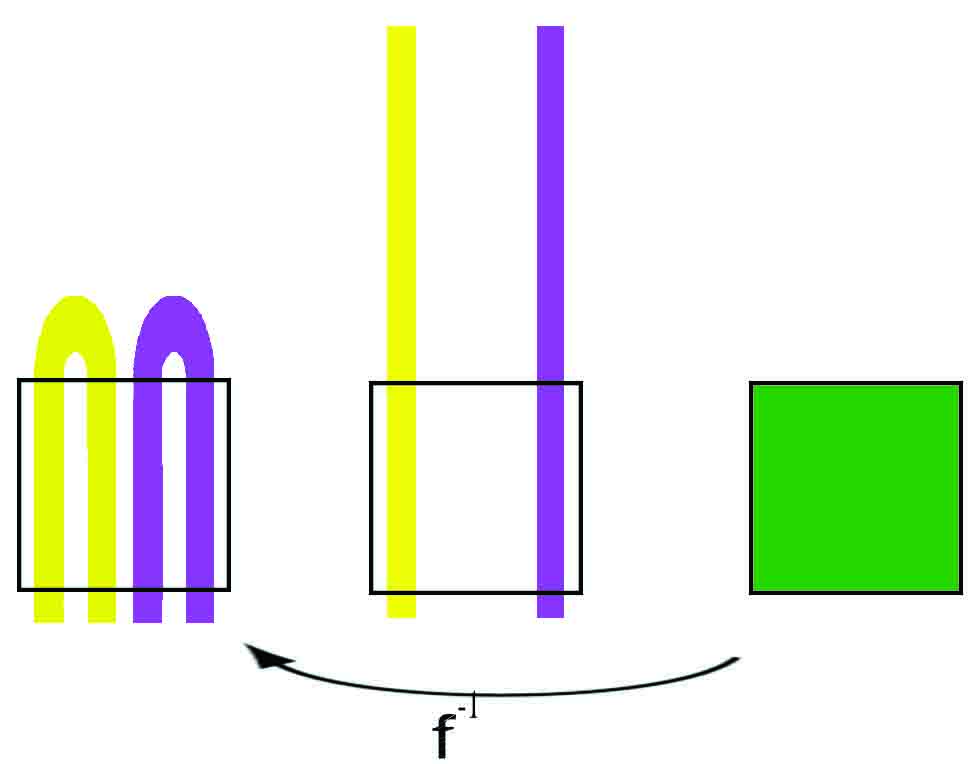}
	\caption{The first backward iterate of the horseshoe map $f$ for $N=2$.}
	\label{2}  
\end{figure}
For any $i$, $1\leq i\leq N$, 
\begin{equation*}%\label{HV}
	f(V_{0^i})=H_a,\quad f(V_{1^i})=H_b. 
\end{equation*}
\begin{figure}[t]
	\includegraphics[width=0.65\textwidth]{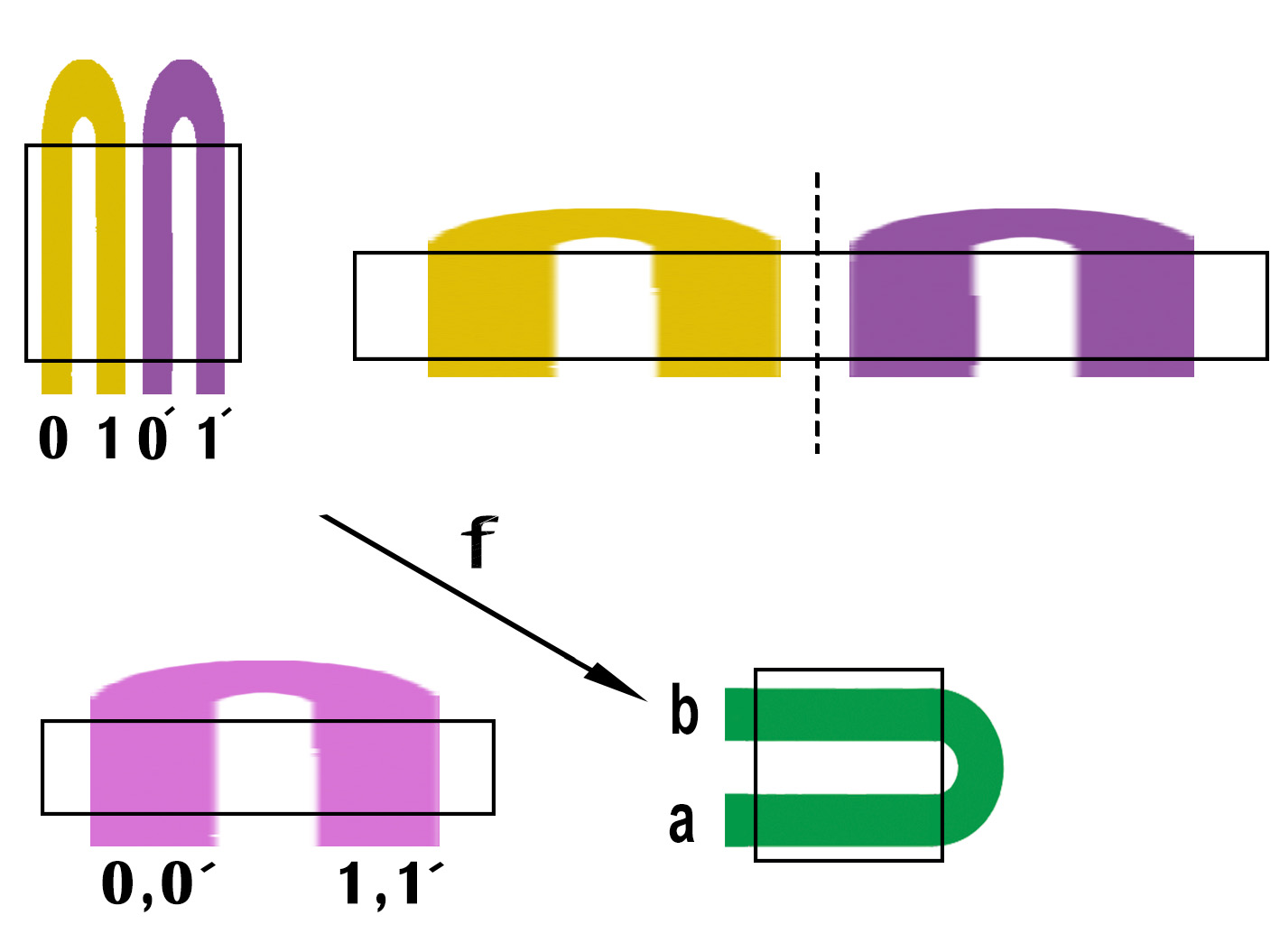}
	\caption{Let  $\mathcal{A}'=\{0,\, 1,\, 0',\, 1'\}$, $\mathcal{A}=\{a,\, b\}$, $\varphi_1(0)=\varphi_1(0')=a$ and $\varphi_1(1)=\varphi_1(1')=b$. First the upper left unit square is mapped to the upper right rectangle with horizontal and vertical sides $\delta$ 
		and $2/\delta$ where $\delta=4+\epsilon$. Cutting the
		rectangle from the dashed line and taking modulo by $2+\frac{\epsilon}{2}$, the rectangles with labels 0 and $0'$ overlay as the ones with labels 1 and $1'$. This gives the lower 
		left rectangle. Bending the obtained rectangle
		back to itself the vertical rectangles are mapped to horizontal rectangle. }
	\label{3}  
\end{figure} 
See Figure \ref{3}. For any $k\geq 2$, $f^{-k}(\textbf{S})\cap \textbf{S} \subset f^{-k+1}(\textbf{S})\cap \textbf{S}$ and
$\lim_{k\rightarrow \infty}f^{-k}(\textbf{S})$ tends to  a vertical line segment in $\textbf{S}$. More precisely, 
$$\Lambda_V=\{(x,\,y)\in \textbf{S}:\ (x,\,y)\in \lim_{k\rightarrow \infty }f^{-k}(\textbf{S})\cap \textbf{S} \}=C_{1}\times [0,\,1],$$
where $C_{1}$ is  isomorphic to a Cantor   set. 
Now let
\begin{equation}\label{Lambda}
	\Lambda=\Lambda_H\cap \Lambda_V=C_{1}\times C_2,
\end{equation}
and define the restriction of of $f$ to $\Lambda$, the horseshoe map. The set $\Lambda$ 
is the trapping region for the horseshoe map $f:\Lambda \rightarrow \Lambda$ and remains invariant under iterations of $f$ and $f^{-1}$.

\subsection{A conjugate zip shift map to the  $N$-to-1   horseshoe}\label{Conj}
To represent this system in symbolic dynamics, we  correspond a bi-infinite sequence of determined symbols to any point in $\Lambda$.
For this purpose, the following observation would be useful.  Let $Z=\{0^1,\,0^2,\,\ldots,0^N\}$,  $O=\{1^1,\,1^2,\,\ldots,1^N\}$ 
and $\mathcal{A}'=Z\cup O$. 
Fix a vertical rectangle $V_{n_0}$, $n_0\in \mathcal{A}'$. 
Let $V_{n_0,n_1}=\{(x,\,y)\in V_{n_0} \,: \, f(x,\,y)\in V_{n_1},\, n_1\in \mathcal{A}'\}=V_{n_0}\cap f^{-1}(V_{n_1})$. Inductively, define
\begin{equation}\label{V}
	\begin{tabular}{rcl}
		$V_{n_0,n_1,\ldots , n_k}$&=&$\{ (x,\,y)\in V_{n_0,\ldots,n_{k-1}}\,:\,\, f^{n_k}(x,\,y)\in V_{n_{k}},\, n_i\in \mathcal{A}'\}$\\
		&=& $V_{n_0,\ldots, n_{k-1}}\cap f^{-n_k}(V_{n_{k}})$.
	\end{tabular} 
\end{equation}
The vertical region $V_{n_0,n_1,\ldots , n_k}$ is the intersection of nested rectangles and tend to the vertical line 
segment $v_{x_0}=\lim_{k\rightarrow \infty}V_{n_0,n_1,\ldots , n_k}$ passing through the point $x_{0}$ on the $x$-axis. Obviously, $v_{x_0}\subset c^s$.

Now consider a horizontal rectangle $H_{n_{-1}}$, $n_{-1}\in \mathcal{A}=\{ a,b \}$. 
Set $H_{n_{-2},n_{-1}}=\{(x,\,y)\in H_{n_{-1}}\,:\, f^{-1}(x,\,y)\in H_{n_{-2}}\}=H_{n_{-1}}\cap f(H_{n_{-2}})$.  Inductively define
\begin{equation}\label{H}
	\begin{tabular}{rcl}
		$H_{n_{-k},\ldots ,n_{-1}}$&=&$\{(x,\,y)\in H_{n_{-(k-1)},\ldots, n_{-1}},\,  f^{-n_k}(x,\,y)\in H_{n_{-k}}\}$\\
		&=& $H_{n_{-(k-1)},\ldots ,n_{-1}}\cap f^{n_k}(H_{n_{-k}})$.
	\end{tabular} 
\end{equation}
Obviously, $\lim_{k\rightarrow \infty} H_{n_{-k},\ldots ,n_{-1}}$ tends to the horizontal line segment $h_{y_0}$ passing through the 
point $y_0$ on the $y$-axis. Clearly, $h_{y_0}\subset c^u$.

Therefore,  
\begin{equation}\label{lim}
	(\lim_{k\rightarrow \infty}V_{n_0,n_1,\ldots , n_k})\cap (\lim_{k\rightarrow \infty} H_{n_{-k},\ldots ,n_{-1}})
\end{equation}
uniquely determines the point $(x_0,\,y_0)=v_{x_{0}}\cap h_{y_0}$, 
where $n_i\in \mathcal{A}'$ for $i\geq 0$ and $n_i\in\mathcal{A}$ for $i\leq -1$.  

The set 
\begin{equation}\label{UP}
	\Sigma_{\mathcal{A},\mathcal{A}'}=\{(n_k)_{k\in \mathbb{Z}}\,:\, (\lim_{k\rightarrow \infty}V_{n_0,n_1,\ldots , n_k})\cap (\lim_{k\rightarrow \infty} H_{n_{-k},\ldots ,n_{-1}})\neq \emptyset\}
\end{equation}
is a zip shift space over alphabets $\mathcal{A}$ and $\mathcal{A}'$. 
We correspond the bi-infinite sequence  $(n_k)_{k\in \mathbb{Z}}$ to a  point $(x_0,\,y_0)\in\Lambda$ with respect to 
the indices obtained in equation \eqref{lim}.

Let $\varsigma :\Lambda\rightarrow \Sigma_{\mathcal{A},\mathcal{A}'}$ be the \textit{code} map which maps each point in $\Lambda$ to its unique
sequence of symbols in $\mathcal{A}\cup \mathcal{A}'$. 
For $z\in Z$ and $o\in O$, define $\varphi_1(z)=a$ and $\varphi_1(o)=b$. Since $\varphi_1$ is not one-to-one, $\varphi_1^{-1}(a)\in Z$ 
and $\varphi_1^{-1}{b}\in O$. 
According to  equations 
\eqref{V} and \eqref{H},  $\varsigma((x_0,\,y_0))=(\ldots,n_{-2},n_{-1}; n_0,n_1,\ldots)$ means that  for $i\in \mathbb{N}$,
\begin{equation}\label{VH}
	f^{-i}(x_0,\,y_0)\cap H_{n_{-i}}\neq \emptyset,\ (x_0,\,y_0)\in V_{n_0}, \ f^{i}(x_0,\,y_0)\in V_{n_{i}}
\end{equation}
where $n_i\in \mathcal{A}'$ for $i\geq 0$ and $n_i\in\mathcal{A}$ for $i\leq -1$.  
Equation \eqref{lim} indicates that for a point $(x_0,\,y_0)\in\Lambda$ all the points on $v_{x_0}\cap \Lambda$ have the same forward code.
But the points on $h_{y_0}\cap \Lambda$ have different backward codes due to the possibility of different choices of pre-images of $x_0$.

The zip shift map on $\Sigma_{\mathcal{A},\mathcal{A}'}$  represents the dynamics of $f$ on $\Lambda$ and also  illustrates  the $n$-to-1 nature of $f$. 
Fix the point $(x_0,\, y_0)\in \Lambda$.   Denote by $\{(x_{-1}^i,\,  y_{-1})\}$ the set of $n$ pre-images of $(x_0,\, y_0)$ under $f$. 
Each of the points $(x_{-1}^i,\,  y_{-1})$ lies on the intersection of $h_{y_{-1}}$ with $v_{x_{-1}^i}$. 
Any segment $v_{x_{-1}^i}$ is a subset of one of the vertical rectangles $\{ V_0^1,\, \ldots, \, V^n_0 \}$ or $\{ V_1^1,\, \ldots, \, V^n_1 \}$.
Since the point $(x_0,\,y_0)$ has $n$ pre-images, there are $n$ possible choices for the place $n_{-1}$ in the corresponding sequence.
Selecting a special branch $(x_{-1}^i,\, y_{-1})$ of pre-images of $(x_0,\, y_0)$, $n_{-1}$ will be $0_i$ or $1_i$ 
with respect to the index of the vertical rectangle $V_0^i$ or $V_1^i$ containing $(x_{-1}^i,\, f^{-1}(y_0))$. 

\begin{theorem}\label{th3.1}
	The factor code map $\varsigma : \Lambda\longrightarrow \Sigma_{\mathcal{A},\mathcal{A}'}$ is a  homeomorphism.
\end{theorem}
\begin{proof}%\marginpar{change the proof}
	By using statements \ref{lim} and \ref{UP} any sequence  $(\ldots,\, ,\,s_{-1};\,s_0,\,s_1,\,s_2,\,\ldots)\in \Sigma_{\mathcal{A},\mathcal{A}'}$
	represents the unique point on the intersection of 
	$v_{x_0}\cap h_{y_0}$ and any such point has a unique coding via $\varsigma$ which shows that $\varsigma$ is one-to-one and onto.
	
	To see the continuity of $\varsigma$, let $\{x_n\}_{i=1}^{+\infty}$ be any sequence of points in $\Lambda$ convergent to some $x\in \Lambda$. By definition of $\varsigma,$ it is obvious that  $\{\varsigma(x_n)\}_{i=1}^{+\infty}$ is convergent to $\varsigma(x).$
	A similar argument shows that $\varsigma^{-1}$ is continuous.

\end{proof}
\begin{theorem}\label{4.2}
	The maps $f:\Lambda\longrightarrow \Lambda$ and $\sigma:\Sigma_{\mathcal{A},\mathcal{A}'}\longrightarrow \Sigma_{\mathcal{A},\mathcal{A}'}$ are topologically
	conjugate via the code map $\varsigma$.
\end{theorem}
\begin{proof}
	According to the  diagram 
	
	$$
	\begin{array}{rrl} \Lambda & \stackrel{f}{\longrightarrow} &
		\Lambda\\
		\varsigma\downarrow && \downarrow\varsigma\\
		\Sigma_{\mathcal{A},\mathcal{A}'}&
		\stackrel{\sigma}{\longrightarrow} &
		\Sigma_{\mathcal{A},\mathcal{A}'},\end{array}
	$$
	we need to show that $f=\varsigma^{-1}\circ \sigma \circ \varsigma$. 
	Let $(x_0,\,y_0)\in \Lambda$  be a point with 
	$$\varsigma((x_0,\,y_0))=(\ldots,\, s_{-2},\,s_{-1};\,s_0,\,s_1,\,\ldots).$$ This means that
	$$\ldots, f^{-1}((x_0,\,y_0))\cap H_{s_{-1}}\neq \emptyset,\ \ (x_0,\,y_0)\in V_{s_0},\ \ f((x_0,\,y_0))\in V_{s_1},\ \ldots.$$
	Equivalently,
	$$\ldots, f^{-1}((x_0,\,y_0))\cap H_{\ldots,s_{-2},s_{-1}}\neq \emptyset,\ \ (x_0,\,y_0)\in V_{s_0,s_1,\ldots},\ \ f((x_0,\,y_0))\in V_{s_1,s_2,\ldots},\ \ldots.$$
	According to the definition of $V_{s_0,s_1,\ldots}$ in \eqref{V}, $V_{s_0,s_1,\ldots}=V_{s_0}\cap f^{-1}(V_{s_1,s_2,\ldots})$. 
	To obtain the code of the point $f((x_0,\,y_0))$, apply $f$ to all of the horizontal and vertical rectangles  containing $(x_0,\,y_0)$. See Figure \ref{4}. 
	\begin{figure}[t]
		\includegraphics[width=0.75\textwidth]{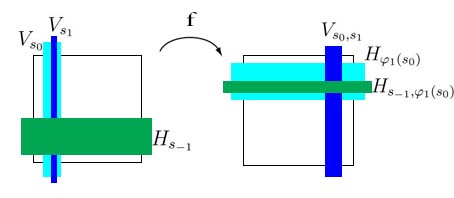}
		\caption{The way $f$ maps vertical strips to horizontal ones.}
		\label{4}  
	\end{figure}
	Then 
	\begin{equation*}%\label{HN}
		(x_1,\,y_1)=f((x_0,\,y_0))\in f(V_{s_0,s_1,\ldots})=f(V_{s_0})\cap V_{s_1,s_2,\ldots}= H_{\varphi_1(s_0)}\cap V_{s_1,s_2,\ldots}.
	\end{equation*}
	Thus, for $k\geq 1$, $f^{n_k}((x_1,\,y_1))=f^{n_k+1}((x_0,\,y_0))\in V_{s_{k}+1,s_k+2,\ldots}$. 
	Therefore, $\varsigma(f(x_0,\,y_0))= (\ldots\, s_{-2},\,s_{-1},\,\varphi_1(s_0);\,s_1,\,s_2\,\ldots)$  
	and $f(x_0,\,y_0)=\varsigma^{-1}\circ \sigma \circ \varsigma(x_0,\,y_0)$ by Theorem \ref{th3.1}.
	\end{proof}
\begin{example}\label{ex:4.3}
	Let $f:\Lambda \rightarrow \Lambda$ be a 2 to 1 horseshoe map defined at the beginning of Section \ref{ho}, where $\Lambda=C_1\times C_2$ as in \eqref{Lambda} for $n=1$.
	Let $\mathcal{A}'=\{1,\, 2,\, 3,\, 4\}$. The usual shift map relevant to this system is the one-sided shift map $\sigma$ over $\Sigma_{\mathcal{A}'}$.

	Suppose $g(x)=4x$ modulo 1. Again its corresponding shift map is $\sigma$ over $\Sigma_{\mathcal{A}'}$. 
	For both systems, the maps $f$ and $g$ are semi-conjugate to $\sigma:\Sigma_{\mathcal{A}'}\rightarrow \Sigma_{\mathcal{A}'}$. 
	Therefore, the shift map cannot distinguish these systems. 
	But if we code them into a zip shift space, then they will have different code spaces as follows.
	
	Let   $\mathcal{A}=\{a,\, b\}$,  $\mathcal{B}=\{a\}$ and  $\mathcal{A}'=\{1,\, 2,\, 3,\, 4\}$ be three sets of alphabets. 
	Then $f$ is conjugate to $\sigma_f:\Sigma_{\mathcal{A},\mathcal{A}'}\rightarrow \Sigma_{\mathcal{A},\mathcal{A}'}$ with $\varphi_1(1)=\varphi_1(3)=a$ and 
	$\varphi_1(2)=\varphi_1(4)=b$. The map $\sigma_f$ illastrates the 2 to 1 feature of $f$. 
	The map $g$ is conjugate to $\sigma_g:\Sigma_{\mathcal{B},\mathcal{A}'}\rightarrow \Sigma_{\mathcal{B},\mathcal{A}'}$ with
	$\varphi_1(1)=\varphi_1(2)=\varphi_1(3)=\varphi_1(4)=a$ which shows that $g$ is a 4 to 1 map.
\end{example}

\section{Stable and unstable sets; periodic, pre-periodic,  homoclinic, heteroclinic points and their orbits}\label{long}
In this subsection  we take $n$ to be 1 in the definition of the transition  map $\varphi_n$ for simplicity. All of the definitions and theorems can be adjusted for $n>1$.
For  points $s=(\ldots, s_{-2},\, s_{-1};\,s_0,\,s_1,\ldots),\, t=(\ldots, t_{-2},\, t_{-1};\,t_0,\,t_1,\ldots) \in \Sigma_{\mathcal{A},\mathcal{A}'}$, 
let  $N(s,\,t)=\min \{ |i|, s_i\neq t_i, i\in\mathbb{Z}\}$,
$N^+(s,\,t)=\min \{ i, s_i\neq t_i, i\geq 0\}$  and $N^-(s,\,t)=\min \{ i+1, s_i\neq t_i, i< 0\}$.
Let $d(s,\,t)=\frac{1}{2^{N(s,\,t)}}$, $d^+(s,\,t)=\frac{1}{2^{N^+(s,\,t)}}$, $d^-(s,\,t)=\frac{1}{2^{N^-(s,\,t)}}$.
Clearly, $d$ is a metric and $d^+$, $d^-$  are pseudo-metrics on $\Sigma_{\mathcal{A},\mathcal{A}'}$. 
For a point $s\in (\Sigma_{\mathcal{A},\mathcal{A}'},\,d)$, $\mathcal{N}_{r}(s)=\{t\in \Sigma_{\mathcal{A},\mathcal{A}'} \,:\, d(t,\,s)<r\}$ is a neighborhood  of
$s$ with diameter $2r$.       
For $i\in \mathbb{Z}$ and $\ell\in \mathbb{N}\cup \{0\}$, the set 
$$s_{i}^{\ell}=[s_i,\, s_{i+1},\ldots, s_{i+\ell}]=\{ (k_n)_{n\in \mathbb{Z}}\in \Sigma_{\mathcal{A},\mathcal{A}'} :\,k_i=s_i,\ldots, k_{i+\ell}=s_{i+\ell}\}$$
is  a \textit{cylinder}. These sets are open and close and form a basis for the topology on 
$\Sigma_{\mathcal{A},\mathcal{A}'}$ induced by the metric $d$. It is easy to see that the space $\Sigma_{\mathcal{A},\mathcal{A}'}$ is 
a Cantor set according to  $d$.

In order to define the notions of stable and unstable sets we equip the 
zip shift space $\Sigma_{\mathcal{A},\mathcal{A}'}$ with an equivalent metric $d^{\pm}(s,\,t)=\frac{1}{2}(d^-(s,\,t)+d^+(s,\,t))$. 
If $\Omega_1$ and $\Omega_2$ are two subsets of $\Sigma_{\mathcal{A},\mathcal{A}'}$, then define 
$d^{\pm}(\Omega_1,\,\Omega_2)=\inf\{d^{\pm}(\omega_1,\,\omega_2):\, \omega_1\in \Omega_1,\,\omega_2\in \Omega_2\}$.
\begin{definition}\label{SU} 
	Let $p=(\ldots, p_{-2},\, p_{-1};\,p_0,\,p_1,\ldots) \in \Sigma_{\mathcal{A},\mathcal{A}'}$ be a periodic point. Then the \textit{stable} and \textit{unstable sets} 
	of $p$ are defined as 
	\begin{align}
		W^s_{\textrm{global}}(p) &=\{ t\in \Sigma_{\mathcal{A},\mathcal{A}'}: d^{\pm}((\sigma^n(p)),\, (\sigma^n(t)))=0,n\rightarrow +\infty \}\nonumber\\
		&= \{t\in \Sigma_{\mathcal{A},\mathcal{A}'}:\ \exists N'(t)\geq 0, \forall n\geq N'(t), t_n=p_n\},\\
		W^s_{\textrm{special}}(p) &=\{ t\in W^s_{\textrm{global}}(p):\ t_n=p_n, \forall n\geq 0 \},\nonumber
	\end{align} 
	and
	\begin{align}
		W^u_{\textrm{global}}(p))&=\{ r\in \Sigma_{\mathcal{A},\mathcal{A}'}: d^{\pm}((\sigma^{-n}(p)),\, (\sigma^{-n}(r)))=0,n\rightarrow +\infty \}\nonumber,\\
		&= \{r\in \Sigma_{\mathcal{A},\mathcal{A}'}:\ \exists N(r)> 0, \forall n\geq N(r), r_{-n}=p_{-n}\},\\
		W^u_{\textrm{special}}(p) &=\{ r\in W^u_{\textrm{global}}(p):\ r_{-n}=p_{-n}, \forall n> 0 \}.\nonumber
	\end{align} 
	\end{definition} 
\begin{remark}
	In the definition of $W^u_{\textrm{global}}(p))$, the metric $d^{\pm}$ is the distance between two sets of points since $\sigma$ is a local homeomorphism. 
	Therefore, for a point $r\in \Sigma_{\mathcal{A},\mathcal{A}'}$,
	the condition 
	$$\{r\in \Sigma_{\mathcal{A},\mathcal{A}'}:\ \exists N(r)> 0, \forall n\geq N(r), r_{-n}=p_{-n}\},$$ is a sufficient condition for $r$ being in $W^u_{\textrm{global}}(p))$.
\end{remark}
For any point $t$, define $D^u(t)=(t_0,\,t_1,\ldots)$ and $D^s(t)=(\ldots, t_{-2},\, t_{-1})$, the stable (forward) and unstable (backward) ``\textit{directions}'' of $t$ respectively.
Let $p=(\overline{p_0,\,p_1,\cdots, p_{m-1}})$ be a periodic point with period $m$ according to Definition \ref{defs}. 
Note that all points $p^i=\sigma^{i}(p)$ for $0\leq i\leq m-1$ are periodic points with period $m$.
They can produce infinitely many stable directions. Howbeit, all of them share the same unstable direction.
In fact, by Definition \ref{SU}, the unstable direction of all points 
$\sigma^{km}(p^i)$, $k\geq 0$, is the same and by equation (\ref{extend})  is
$D^u(p^i)=\{ (\varphi_1(\overline{p_i,\cdots,p_{m-i-1}})): 0\leq i\leq m-1\}$. Denote by $D^u(O(p))$ the union of the unstable directions of $p^i$
for $0\leq i\leq m-1$ i.e., $$D^u(O(p)):=\bigcup_{i=0}^{m-1} D^u(p^i).$$ 

If $\sigma$ fails to be one to one on $p$, then $p$ has pre-images as
	\begin{equation}\label{pr}
		(\varphi_1(\overline{p_{m-1},p_0,\cdots,p_{m-2}}); \varphi_1^{-1}(\varphi_1(p_{m-1})), \overline{p_0,\,p_1, \cdots, p_{m-1}}),
	\end{equation}
	where obviously, $p_{m-1}\in \varphi_1^{-1}(\varphi_1(p_{m-1}))$.

	\begin{definition}\label{pre-p}
		Let $p=(\overline{p_0,\ldots, p_{m-1}})$ be a periodic point of period $m$ and  
		let $O_n(p)=f^{-n}(O(p))$. Then we call $p'\in O_n(p^i)$ the pre-periodic point of $p^i=\sigma^{i}(p)$ if there exists some $k\geq 1$ such that $\sigma^{km}(p')=p^i$
		(See Figure \ref{PP}).
	\end{definition}

	\begin{figure}
\includegraphics[width=0.75\textwidth]{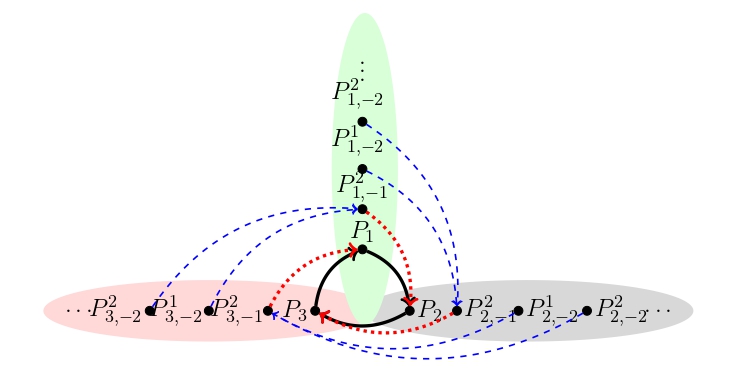}		
		\caption{The 1st and 2nd pre-periodic points for periodic points of period 3. }\label{PP}
	\end{figure}
	\vspace{0.25cm}
	\begin{definition}\label{hohe}
		Let $p=(\overline{p_0,\ldots, p_{m-1}})$ and  
		$q=(\overline{q_0,\ldots, q_{\ell-1}})$ 
		be  periodic points in $\Sigma_{\mathcal{A},\mathcal{A'}}$. Then $x\in \Sigma_{\mathcal{A},\mathcal{A'}}$ is a \textit{homoclinic}  point to $p$, 
		if $x\in W^s_{\textrm{special}}(p)\cap W^u_{\textrm{special}}(p)$ or equivalently, if 
		there exists integers $N(x)>0$ and $N'(x)\geq 0$ such that for $n\geq N(x)$, $x_{-n}=p_{-n}$ and for $n\geq N'(x)$, $x_{n'}=p_n'$. Also, 
		$y\in \Sigma_{\mathcal{A},\mathcal{A'}}$ is a \textit{heteroclinic}  point, if $y\in W^s_{\textrm{special}}(p)\cap W^u_{\textrm{special}}(q)$.
	\end{definition} 
The orbit of homoclinic points is not necessarily unique. 
Let $p$, $N$ and $N'$ be as in  Definition \ref{hohe}. For some $k, k'\geq 0$, let 
$$x=(\overline{\varphi_1(p_k),\ldots, \varphi_1(p_{m-1+k})}, x_{-N+1},\, \ldots, x_{-1};\, x_0,x_1,\ldots, x_{N'-1},\overline{p_{k'},\ldots,p_{m-1+k'}})$$

be a homoclinic point to $p$ (the indices  are taken  modulo by $m$). 
Denote the orbit of $x$ by points $h_i(p,x)$, $i\in \mathbb{Z}$. To determine the $h_i(p,x)$ consider the following 3 cases. 
\begin{itemize}
	\item If $i\geq 0$, then $h_i(p,x)=\sigma^{i}(x)$.\\
	\item For $i=-N-r$ with $r\geq0$, let $h_{i}(p,x)=(y_j)_{j\in \mathbb{Z}}$. 
	If $j\leq r$ or $j\geq N+r$, the entries of $(y_j)_{j\in\mathbb{Z}}$ can be easily determined according to the periodic point $p$. 
	It means that 
	$$y_j=\left\{\begin{tabular}{ll}
		$\varphi_1(p_{m-1+k-r+j})\hspace{7mm} $& if $j<0$\\
		$p_{m-1+k-r+j}\hspace{7mm} $& if $0\leq j\leq r$\\
		$x_{j-N-r}$& if $j\geq N+r$,
	\end{tabular}\right.$$ 
	and $y_{j}\in \varphi^{-1}_1(x_{j-N-r})$ if $r<j<N+r$. All the indices can be uniquely determined except when $r<j<N+r$. 
	For entries with these indices, we should choose ``appropriate" elements of $\varphi^{-1}_1(x_{j-N-r})$. 
	By appropriate element, we mean an element  $y_j$ such that $(y_j)_{j\in\mathbb{Z}}\in  \Sigma_{\mathcal{A},\mathcal{A}'}$ 
	and also, for any $i=-N-r$, $r\geq 0$ the corresponding $(y_j)_{j\in\mathbb{Z}}$ should belong to $\Sigma_{\mathcal{A},\mathcal{A}'}$.
	The number of these orbits is less than or equal to $\sum_{t=-N+1}^{-1}|\varphi^{-1}_1(x_t)|$.\\
	\item If $-N<i<0$, then $h_i(p,x)=\sigma^{N+i}(h_{-N}(p,x))$.	
\end{itemize}
Obviously, the homoclinic point $x$ might have more than one homoclinic orbit to $p$.
See Example \ref{eg3.5}.
The similar procedure works to find heteroclinic orbits.
\subsection{Stable and unstable sets and periodic , homoclinic and heteroclinic  orbits of the $N$-to-1 horseshoe}\label{PPH-ho}
In this subsection, some properties of $f$ such as  stable and unstable sets, periodic, pre-periodic, homoclinic,  heteroclinic   points and their  orbits are verified.
Here we take $n$ to be 1 ($n$ is due to $\varphi_n$) for simplicity.

\subsubsection{ Perodic  points and their orbits} Consider the periodic point $P=(x_0,\, y_0)$ and its
orbit $(x_i,\,y_i)\in \Lambda$, $0\leq i\leq m-1$ 
with least period $m$. Since $f$ is an $N$-to-1 endomorphism, 
any point $(x_i,\,y_i)$ has $N$ pre-images. To have a periodic orbit, we
take $f^{-1}((x_i,\,y_i))$ to be $(x_{i-1},\, y_{i-1})$ for $1\leq i\leq m-1$ 
and $f^{-1}((x_0,\,y_0))=(x_{m-1},\, y_{m-1})$. 
Other pre-images of $f$ do not form a periodic orbit.
Let $\varsigma((x_0,\,y_0))=p=(\overline{p_0,\, p_1,\, \ldots,\,p_{m-1}})$, $p_i\in\mathcal{A}'$. 
Then for $0\leq i\leq m-1$, $\varsigma((x_i,\,y_i))=p^i=(\overline{p_i,\, p_{i+1},\, \ldots,\,p_{i+m-1}})$ where the indices  are taken  modulo by $m$.
Similarly, $p^i$ has $N$ pre-images.

Any point $\varsigma(x_i,\,y_i)=(\ldots,\varphi_1(p_{i-1});\, p_i,\ldots,p_{m-1},\,p_0,\,p_1,\ldots)$ 
has $N$ pre-images as $(\ldots,\varphi_1(p_{i-2})\, ;t,\, p_i,\ldots,p_{m-1},\,p_0,\,p_1,\ldots)$ with $t\in \varphi_1^{-1}(p_{i-1})$  for $1\leq i\leq m-1$. 
To have the periodic orbit  select $t$ to be $p_{i-1}$. Then  
$$(\ldots,\varphi_1(p_{i-2}); \, p_{i-1},\,  p_i, \ldots, p_{m-1},\,p_0,\,p_1,\ldots)$$
coincides with the code of that pre-image of $f^{-1}(x_i,\,y_i)$ which belongs to the set $\{(x_0,\,y_0),\ldots,(x_{m-1},\,y_{m-1})\}$.
None of the other pre-images of $p^i$ are periodic.
% The total number of periodic points of least period $m$ is $(2n)^m$.
Let $\textrm{Per}_m(f)\subset \Lambda$ and $\textrm{Per}_m(\sigma)\subset \Sigma_{\mathcal{A},\mathcal{A}'}$ be the set of all periodic points  with least period $m$
under $f$ and $\sigma$ respectively. Obviously, $|\textrm{Per}_m(\sigma)|=(2N)^m$ and the 
set $\textrm{Per}(\sigma)=\cup_{m\geq 0}\textrm{Per}_m(\sigma)$ is dense in $\Sigma_{\mathcal{A},\mathcal{A}'}$.
According to the Theorem \ref{th3.1}, the same result holds for $f$ over $\Lambda$.

\subsubsection{Stable and unstable sets}

Consider a point $S_0=(x_0,\, y_0)\in \Lambda$. As mentioned in Subsection \ref{Conj}, it lies on the intersection of line segments 
$h_{y_0}$ and   $v_{x_0}$ passing through $x_{0}$ and $y_{0}$ respectively.

In the sense of symbolic dynamics,  let 
$\varsigma(S_0)=s=(\ldots,s_{-2},\, s_{-1};\, s_0,\, s_{1}, \ldots)$. 
Therefore, by Definition \ref{SU}, $t=(\ldots,t_{-2},\, t_{-1};\, t_0,\, t_{1}, \ldots)\in W^s(s)$ if there exists 
an integer $N'(t)\geq 0$
with $t_n=s_n$ for $n\geq N'(T)$. Also, $r=(\ldots,r_{-2},\, r_{-1};\, r_0, \ldots)\in W^u(s)$  if  there exists an integer $N(r)>0$
with $r_{-n}=s_{-n}$ for $n\geq N(r)$.  In fact since the zip shift space representing the $N$-to-1 horseshoe is a full zip shift, for any letter in 
$\mathcal{A}=\{a,\,b\}$, $\varphi^{-1}\{a\}$ or $\varphi^{-1}\{b\}$ can choose any symbol in $Z$ or $O$ without restrictions respectively. 
For non full zip shift cases, adjacent entries impose restrictions on the choice of symbols from  $Z$ or $O$  
to avoid the occurrence of forbidden words. Therefore in full zip shift spaces, the statement 
$d^{\pm}((\sigma^{-n}(r)),\, (\sigma^{-n}(s)))=0,n\rightarrow +\infty$ is equivalent to $r_{-n}=s_{-n}$ for $n\geq N(r)$. 

Consider the vertical line segment $v_{x_0}$. All points $(x_0,\,y)$ on $v_{x_0}\cap \Lambda$ have the code 
$\varsigma((x_0,\,y))\in \{t=(\ldots,t_{-2},\, t_{-1};\, t_0,\, t_{1}, \ldots), t_n=s_n, n \geq 0 \}\subset W^s(s)$. 
These  points  have the same forward code. If $T\in v_{x_0}$, then $d(f^n(S_0),\, f^n(T))\rightarrow 0$ when $n\rightarrow +\infty$.
Therefore, for $S_0\in v_{x_0}\cap \Lambda$, we define
\begin{equation}\label{H}
	\begin{aligned}
		W^s_{\textrm{special}}(S_0) &:= v_{x_0} \\
		&\supset \{ T \in \mathbf{S} : \varsigma(T) = (\ldots, t_{-2}, t_{-1}; t_0, t_1, \ldots), \ t_n = s_n, \ \mathbf{n} \geq \mathbf{0} \}, \\
		W^s_{\textrm{global}}(S_0) &:= \left\{ T \in \bigcup_{k \geq 0} f^{-k}(v_{x_0}) : k \in \mathbb{N} \cup \{0\} \right\} \\
		&\supset \{ T \in \mathbb{R}^2 : \varsigma(T) \in W^s_{\textrm{global}}(s) \}.
	\end{aligned}
\end{equation}
%\marginpar{global is defined only for periodic p}
For $T\in W^s_{\textrm{global}}(S_0)$, there exists $N'=N'(T)\geq 0$ such that for $n\geq N'$, $d(f^n(S_0),\, f^n(T))\rightarrow 0$ when $n\rightarrow +\infty$. 

For the horizontal line segment $h_{y_0}$, all points $(x,\,y_0)$ on $h_{y_0}\cap \Lambda$ have the code 
$\varsigma((x,\,y_0))\in \{r=(\ldots, r_{-2},\, r_{-1};\, r_0,\, r_{1}, \ldots),\ r_{-n}=s_{-n},\, n>0\}\subset W^u_{\textrm{special}}(s)$. 
Such points  have the same backward code. Using the notations of the Subsection \ref{Conj}, the point $S_0=(x_0,\,y_0)$ on $h_{y_0}$ has $N$ 
pre-images denoted by $\{(x^i_{-1},\,y_{-1}): 1\leq i\leq N\}$. 
All these points lie on
the unique horizontal line segmant $h_{y_{-1}}$ but on different stable line segments passing through $x^i_{-1}$. 
In fact,  $f^{-1}(h_{y_0})\subset h_{y_{-1}}$. Therefore, the unstable direction of all points on $f^{-1}(y_0)$ overlay. We define 
\begin{equation}\label{H}
	\begin{aligned}
		W^u_{\textrm{special}}(S_0) &= h_{y_0} \\
		&\supset \{ R \in \mathbf{S} : \varsigma(R) = (\ldots, r_{-2}, r_{-1}; r_0, r_1, \ldots), \ r_{-n} = s_{-n}, \ \mathbf{i} > \mathbf{0} \}, \\
		W^u_{\textrm{global}}(S_0) &= \{ R \in f^{N}(h_{y_0}) : N \in \mathbb{N} \} \\
		&\supset \{ R \in \mathbb{R}^2 : \varsigma(R) \in W^u(s) \}.
	\end{aligned}
\end{equation}
If $R\in W^u_{\textrm{global}}(S_0)$, then there exits $N=N(R)>0$ such than the $d^{\pm}$ distance between  the set of points $f^{-N}(S_0)$ 
and the set of points   $f^{-N}(R)$ tends to zero  as $n\rightarrow +\infty$.

The unstable curve and stable curves can also be determined by symbolic labels. 
As mentioned before, $f^{-1}(\mathbf{S})$ is consisted of $N$ vertical strips and in all of them, the rectangle $V_{1^i}$ is
followed by $V_{0^i}$ for $1\leq i\leq N$. This gives a rule denoted  by 
\begin{equation}\label{s1}
0_1\rightarrow 1_1\quad 0_2\rightarrow 1_2, \ldots, 0_N\rightarrow 1_N.
\end{equation} 
For simplicity, take $N=2$. As in Figures \ref{2} and \ref{3},  in the first step, the rectangle $V_1$ is followed by $V_0$ and also ${V}_{1'}$ is
followed by ${V}_{0'}$. Therefore,
\begin{equation}\label{s1}
0\rightarrow 1\quad\quad 0'\rightarrow 1'.
\end{equation} 

In the second backward iterate, $f^{-2}(\mathbf{S})$ is consisted of 4 vertical strips. 
By following the labels of the rectangles in each strip, the rules in \eqref{s1} extend to
\begin{equation}\label{s2}
\begin{tabular}{rcl}
	$ 00\rightarrow 10\rightarrow 11\rightarrow 01$ &$\quad $& $0'0\rightarrow 1'0\rightarrow 1'1\rightarrow 0'1$\\
	$00'\rightarrow 10'\rightarrow 11'\rightarrow 01'$ &$\quad$ & $0'0'\rightarrow 1'0'\rightarrow 1'1'\rightarrow 0'1'$.
\end{tabular}  
\end{equation} 
Note that the first item in (\ref{s2}) means that the rectangles $V_{00}$, $V_{10}$, $V_{11}$  and $V_{01}$ are followed one after another
in $f^{-2}(\mathbf{S})$. 
The index $ij$ in $V_{ij}$ means that the rectangle $V_j$ has stretched  vertically into the rectangle $V_i$ under $f^{-1}$. 
In fact, the strings in \eqref{s1} produce the stable strings in \eqref{s2}. For example, the strings $ 00\rightarrow 10\rightarrow 11\rightarrow 01$ 
and $0'0\rightarrow 1'0\rightarrow 1'1\rightarrow 0'1$ are constructed 
from $0\rightarrow 1$ in (\ref{s1}) in the following sense. 

The first line in (\ref{s1}) means $V_0$ has stretched into $V_0$ and $V_1$ to create $V_{00}$ followed by $V_{10}$ and then $V_1$ 
has stretched in $V_1$ and $V_0$ to create $V_{11}$ followed by $V_{01}$ (Note that the order of rectangles $V_0$ and $V_1$ is changed when $V_1$ is stretching  into them). 
Similarly, $V_0$  has stretched in $V_{0'}$ and $V_{1'}$.
Afterwards, $V_1$  has stretched in $V_{1'}$ and $V_{0'}$ to form the connected strip of rectangles 
$V_{0'0}$, $V_{1'0}$, $V_{1'1}$ and $V_{0'1}$. 
Therefore, $0\rightarrow 1$ turns to two sequences. Each sequence has 4 words $t={t'}_1{t'}_2$ of length two. The letters ${t'}_2$ in the first (resp. second) two words 
in both sequences equal to 0 (resp. 1) which means $V_0$ (resp. $V_1$) in $0\rightarrow 1$ has been stretched. Because  of the way the horseshoe is constructed, 
the first two letters ${t'}_1$ would be $0$ then $1$ (resp. $0'$ then $1'$) and the second two letters ${t'}_1$ would be $1$ then $0$ (resp. $1'$ then $0'$).
Likewise, the two other strings in \eqref{s2} are formed. 

In the $k$th step, there are $2^k$ disjoint stable strings of length $2^k$ and its entries belong to $\beta_k'$. 
If a word $w=s_1's_2'\ldots s_k'\in \beta_k'$ is given, we can simply find all of the vertical rectangles which are connected 
to $V_{s'_1\ldots\,s'_k}$ in the stabel strip containing $V_w$ in $f^{-k}(\mathbf{S})$. 
To do this consider a sequence containing $2^k$ entries of words such as  $t={t'}_1\ldots {t'}_k$. The words in this sequence resemble the labels  of a sequence of
connected vertical rectangles in the stable strip containing the rectangle with lable $s'_1 s'_2 \ldots s'_k$.
If $s'_k \in \{0,\,1\}$, then the first $2^{k-1}$ words get $0$ for $t'_k$ and the other words get $1$ for $t'_k$.  If $s'_k \in \{0',\,1'\}$, simply replace $0$ and $1$ by 
$0'$ and $1'$. Continuing this procedure, for any $i\in \{1\ldots k\}$, if $s'_i \in \{0,\,1\}$, then the first $2^{i-1}$ words in 
the string get $0$ for $t'_i$, the  next $2\times 2^{i-1}$ words receive $1$, the  next  $2\times 2^{i-1}$ words receive $0$ and so on. The last $2^{i-1}$ words receive $0$ for $t'_i$.
If $s'_i \in \{0',\,1'\}$,  replace $0$ and $1$ by 
$0'$ and $1'$. 

For example let $w=s_1's_2' s_3'=10'1'$. So, $s'_1=1$, $s'_2=0'$ and $s'_3=1'$. 
The connected stable string to $w$ has length $2^3$ and its entries are in $\beta'_3$. 
Since $s'_3=1'$, the first $2^2$ words in the string receive $0'$ and the other words get $1'$ in their 3rd positions. 
To determine the letters in the second position of the words, since $s'_2=0'$, the first two words receive $0'$, 
the next four words get $1'$ and the last two words get $0'$. 
Finaly, $s'_1=1$. The first position of the words in the string should be 0, 1, 1, 0, 0, 1, 1 and 0 respectively. Therefore, the   connected stable string to $w$
is as follows:
$$ 00'0'\rightarrow 10'0'\rightarrow 11'0'\rightarrow 01'0' \rightarrow 01'1' \rightarrow 11'1'\rightarrow w=10'1'\rightarrow 00'1'. $$
Note that any word $w\in \beta'$ belongs  to a unique stable strip and can be used as the label of that strip.

Due to these strips, we define
\begin{equation}
	\begin{aligned}
		W^s(P) &= \{ R \in c_{x_0}^s : f^{km}(R) \subset v_{x_0}, \ k \geq 0 \}, \\
		W^u(P) &= \{ R \in c_{y_0}^u : f^{-km}(R) \subset h_{y_0}, \ k \geq 0 \}.
	\end{aligned}
	\label{eq:stable_unstable}
\end{equation}	
\subsubsection{Homoclinic and heteroclinic points and their orbits} Without loss of generality, suppose that $f$ is a 2-to-1 horseshoe map.
Let $P=(x_0,\,y_0)\in \Lambda$ be a periodic point with zip shift code $\varsigma(P)=p=(\overline{p_0,\,p_1,\ldots, p_{m-1}})$. 
By \eqref{H}, $P$ lies on the intersection of stable line segment 
$v_{x_0}\cap \Lambda$ and unstable line segment
$h_{y_0}\cap \Lambda$ consisting of points with coding $(\ldots; \overline{p_0,\,p_1,\ldots, p_{m-1}})$ and $(\varphi_1(\overline{p_0,\,p_1,\ldots, p_{m-1}});\ldots),$ respectively.
Therefore, any point $P'$ in  $W^s(P)\cap W^u(P)$  is a homoclinic point to $P$. 
If $\varsigma(P')=r\in \Sigma_{\mathcal{A},\mathcal{A}'}$, then there are natural numbers 
$N(r)>0$ and $N'(r)\geq 0$ such that $p_n=r_n$ for $n\geq N(r)$ and $p_{-n}=r_{-n}$ for $n\geq N'(r)$.
Let $P=(x_0,\,y_0)$ be a fixed point. 
For any $k\geq 0$, $f^k(P)=P$ and all the forward orbits of $P$ lie on $h_{y_0}$. 
The point $P$ has two pre-images under $f$, say $P$ and
$P'$ where $P'$ is a pre-fixed point. 
One of them passes through $P$
and the other passes through $P'$. After $k$ backward iterates, $2^k$ stable curves passes through the pre-images of $P$.

	Let  $H=(x,\,y)\in \Lambda$ be a homoclinic point to $P$.
	Denote by $h_i(P,\, H)$, the point in $i$th position in homoclinic orbit of $(x,\,y)$ to $P$ for $i\in \mathbb{Z}$. 
	Therefore, the homoclinic orbit of $(x,\,y)$ to $P$ is the set of points
	$$P\, \leftarrow \ldots, h_{-2}(P,\, H),\, h_{-1}(P,\, H),\, H,\, h_1(P,\, H),\, \ldots \rightarrow P.$$
	Let $u=\varsigma(H)=(\ldots, u_{-2}, u_{-1}; {u'}_{0}, {u'}_1, \ldots  )$. Then $u$ is a homoclinic point to $p$. Let $h_i(p,\,u)$ be the homoclinic orbit of $u$ to $p$.
	This orbit can be determined as the argument in Section \ref{long}.
		$\mathcal{A}=\{a,\, b\}$,  $\mathcal{B}=\{a\}$ and  $\mathcal{A}'=\{1,\, 2,\, 3,\, 4\}$ be three sets of alphabets. 
		Then $f$ is conjugate to $\sigma_f:\Sigma_{\mathcal{A},\mathcal{A}'}\rightarrow \Sigma_{\mathcal{A},\mathcal{A}'}$ with $\varphi_1(1)=\varphi_1(3)=a$ and 
		$\varphi_1(2)=\varphi_1(4)=b$. The map $\sigma_f$ illastrates the 2 to 1 feature of $f$. 
		The map $g$ is conjugate to $\sigma_g:\Sigma_{\mathcal{B},\mathcal{A}'}\rightarrow \Sigma_{\mathcal{B},\mathcal{A}'}$ with
		$\varphi_1(1)=\varphi_1(2)=\varphi_1(3)=\varphi_1(4)=a$ which shows that $g$ is a 4 to 1 map.

\begin{example}\label{eg3.5}
	Let  $\mathcal{A}'=\{0,\, 1,\, 0',\, 1'\}$, $\mathcal{A}=\{a,\, b\}$, $\varphi_1(0)=\varphi_1(0')=a$ and $\varphi_1(1)=\varphi_1(1')=b$ as in Figure \ref{3}.  
	Let $H$ be the homoclinic point to a periodic point $P$ with coding $\varsigma (P)=(\overline{01'})$ and $\varsigma (H)=h= (\overline{a,b},b,a; 1,\overline{0,1'})$. 
	Obviously, the orbit of $P$ has two stable directions. To obtain the backward orbit of $H$ to $P$ we use the second approach since $\Sigma_{\mathcal{A},\mathcal{A}'}$
	is a full zip shift. Let $\sigma^{-1}(h)\in\{h_1,h_2\}$ where $h_1=(\overline{a,b},b;0,1,\overline{1,0})$ and
	$h_2=(\overline{a,b},b;0',1,\overline{1,0})$. Also, 
	$\sigma^{-1}(h_1)\in\{h_{11},h_{12}\}$ and $\sigma^{-1}(h_2)\in\{h_{21},h_{22}\}$ where 
	$h_{11}=(\overline{a,b};1,0,1,\overline{1,0})$, $h_{12}=(\overline{a,b};1',0,1,\overline{1,0})$, $h_{21}=(\overline{a,b};1,0',1,\overline{1,0})$ 
	and $h_{22}=(\overline{a,b};1',0',1,\overline{1,0})$. 
	Now, since we have reached to the repeating part of the periodic point, the pre-images are determined according to the point $P$. 
	This means that we choose $\varphi_1^{-1}(a)$ and $\varphi_1^{-1}(b)$ to be $0$ and $1$ respectively. Therefore, $\sigma^{-1}(h_{ij})$ 
	would be $(\overline{b,a};1, \ldots)$ where the doted part is the entries of  non-negative indices  of $h_{ij}$, $i,j\in\{1,2\}$.
	The homoclinic point $H$ produces 4 homoclinic orbits. All of these orbits converges to $P$ in the past and future.
\end{example}	

	\bibliographystyle{amsplain}
	
	\end{document}